\documentclass[12pt]{article}

\usepackage[english]{babel}
\usepackage[a4paper,top=2cm,bottom=2cm,left=3cm,right=3cm,marginparwidth=1.75cm]{geometry}
\usepackage{xcolor}
\usepackage{amsmath,amsfonts,amsthm,amssymb,mathrsfs,bm}
\usepackage{graphicx}
\usepackage{booktabs}
\usepackage{caption}
\usepackage[colorlinks=true,allcolors=blue]{hyperref}
\usepackage{authblk} 

\newtheorem{theorem}{Theorem}

\newtheorem{lemma}{Lemma}

\newtheorem{definition}{Definition}

\newtheorem{remark}{Remark}
\newtheorem{notation}{Notation} 

\newenvironment{acknowledgements}{\section*{Acknowledgements}}{}
\newenvironment{funding}{\section*{Funding}}{}

\title{Least zero of pairs of additive cubic equations}
\author[1]{Yixiu Xiao}
\author[2,*]{Hongze Li}
\affil[1]{School of Mathematical Sciences, Shanghai Jiao Tong University, Shanghai, China\\
\texttt{asaka1312@sjtu.edu.cn}}
\affil[2]{School of Mathematical Sciences, Shanghai Jiao Tong University, Shanghai, China\\
\texttt{lihz@sjtu.edu.cn}}
\date{} 

\begin{document}
\maketitle
\abstract
{An effective upper bound is established for the least non-trivial integer solution to the system of cubic forms 
\begin{equation}
\label{eq:zuikaishi}
     \begin{cases}
         F=c_{1}x_1^3+c_{2}x_2^3+\cdots+c_{n}x_n^3=0,\\
         G=d_{1}x_1^3+d_{2}x_2^3+\cdots+d_{n}x_n^3=0,\\
     \end{cases}
\end{equation}
under the "$M$-good" condition for \( n \geq 16 \), where \( c_{1}, \dots, c_{n} \) and \( d_{1}, \dots, d_{n} \) are integers. Additionally, a range is derived for the probability that  randomly selected simultaneous equations satisfy the $M$-good condition.}



\maketitle

\section{Introduction}

Davenport and Lewis \cite{Davenport1966CubicEO} have shown that \eqref{eq:zuikaishi} is soluble (with not all of $x_1,...,x_n$ zero) for every $n \ge 18$. Cook \cite{Cook1972PairsOA} has shown that 18 can be replaced by 17. Then R.C.Vaughan \cite{Vaughan_Additive_1975} has shown that the value of n needed for general equations can be reduced to 16.
 An example due to Davenport and Lewis \cite[(4) and (5)]{Davenport1966CubicEO} with $n=15$ has no non-trivial solution; this shows that $n=16$ is essentially the best possible.

It is very natural to find an effective upper bound for the smallest positive integer $\lambda$ with the property that when there is a non-trivial solution $\mathbf{x} = (x_1,...,x_n) \in \mathbb{Z}^{n}$ to \eqref{eq:zuikaishi}, so there is such a solution with $\max\{\lvert x_{i} \rvert
: 1\leq i \leq n \} \leq \lambda$. We denote this quantity by $\Lambda_{n}(F,G)$, simply written as $\Lambda_{n}$.

Let $F,G \in \mathbb{Z}[X_1,...,X_n]$ be diagonal cubic forms in \eqref{eq:zuikaishi}, with coefficients of maximum modulus $\left\|F\right\|$, $\left\|G\right\|$. We write
\begin{equation}
    M(F,G) = \max\{\left\|F\right\|,\left\|G\right\|\}, 
\end{equation}
simply written as $M$.

\begin{definition}
\label{def_M_good}
    We call the simultaneous equations \eqref{eq:zuikaishi} "$M$-good" if the set of coefficients
    \begin{equation}
    \label{set_S_definition}
        S = \{ c_i d_i^{-1} \mid i=1, \ldots, n\}
    \end{equation}
    satisfies
    \begin{itemize}
        \item [(i)] For all primes $p \leq M^2$ with $p \equiv 1 \pmod{3}$, there are at most nine identical elements in $S \pmod{p}$;
        \item [(ii)] For all primes $p \leq M^2$ with $p \equiv 2 \pmod{3}$, there are at least three distinct ratios in $S \pmod{p}$;
        \item [(iii)] There are  at least three distinct ratios in $S \pmod{3}$.
    \end{itemize}
    Here, $d_i^{-1}$ denotes the inverse of $d_i$ in $\mathbb{F}_p$ if $d_i \not\equiv 0 \pmod{p}$; the ratio "$\frac{0}{0}$" is identified with any existing ratios and does not introduce new types.
\end{definition}

The purpose of this paper is to establish the following theorems.

\begin{theorem}
\label{zhudingli}
Let $F,G \in \mathbb{Z}[X_1,...,X_n]$ be diagonal cubic forms in \eqref{eq:zuikaishi}, with $n \geq 16$. Then there exists a constant $c > 0$ such that 
$$
\Lambda_{n} \leq cM^{2411} ,
$$
whenever the two simultaneous equations \eqref{eq:zuikaishi} are  $M$-good. 
\end{theorem}

\begin{theorem}
\label{cidingli}
Let $F,G \in \mathbb{Z}[X_1,...,X_n]$ be diagonal cubic forms in \eqref{eq:zuikaishi}, with $n \geq 16$. And the coefficients are chosen randomly. Then the probability that \eqref{eq:zuikaishi} is $M$-good satisfies
\begin{equation*}
    0.9694\leq \text{Prob}_\text{$M$-good} \leq  0.9700.
\end{equation*}
\end{theorem}

\begin{notation}
The Vinogradov symbols $\ll , \gg$ have their usual meanings, namely that for functions f and g with g taking non-negative real values, $f \ll g$ means $|f| \leq Cg$ for some implied constant $C$.  $e(\alpha)=e^{2\pi i\alpha}$. We shall use $\varepsilon$ to denote a sufficiently small positive number, and the symbols $c$ or $c_{n, \varepsilon}$ to denote various positive real constants, not necessarily the same at each occurrence. All of the implied constants in our work will be allowed to depend on $\varepsilon$ and $n,$ with any further dependence being made completely explicit. 
We write $|\mathbf{x}|$ for the norm $\max\{\lvert x_{i} \rvert
: 1\leq i \leq n \}$ of any vector $\mathbf{x} = (x_1,...,x_n) \in \mathbb{R}^{n}$.
\end{notation}

\section{The rearrangement of the variables}

We may assume that for each \( j \), at least one of \( c_j \) and \( d_j \) is non-zero; otherwise, the problem becomes trivial. Two ratios \( c_i/d_i \) and \( c_j/d_j \) are said to be equal if \( c_i d_j = c_j d_i \). For brevity, we denote \( r_i = c_i/d_i \), where \( r_i = \infty \) is permitted. Since \( \Lambda_n \) is non-increasing in \( n \), it suffices to restrict our attention to the case \( n = 16 \). Unless specified otherwise, all subsequent discussions in this paper will be confined to \( n = 16 \).

We first consider the relatively simple case where at least seven of the ratios \( r_i \) are equal. The following lemma is applicable to this scenario.

\begin{lemma}
\label{lemma2.1}
Suppose that there is a ratio repeated at least seven times among the ratios $r_{1},...,r_{16}$. Then
$$\Lambda_{n} \ll M^{\frac{98}{3}} .$$
\end{lemma}

\begin{proof}
Without loss of generality, we suppose that $r_{1}=\cdots=r_{7}$. Then it follows from an argument of Hongze Li \cite[Theorem 1]{MR1460578} that 
$$ c_{1}x_1^3+c_{2}x_2^3+\cdots+c_{7}x_7^3=0 $$
has a non-trivial solution in the range $|\mathbf{x}| \ll 
\left(\prod \limits_{i=1}^7|c_{i}|\right)^{\frac{14}{3}}$, say $\mathbf{v}$.
Naturally, $\mathbf{v}$ is also a solution of 
$$ d_{1}x_1^3+d_{2}x_2^3+\cdots+d_{7}x_7^3=0 .$$
Thus $\tilde{\mathbf{v}}=(\mathbf{v},\mathbf{0})$ is a solution of \eqref{eq:zuikaishi} with $0 < |\tilde{\mathbf{v}}|=|\mathbf{v}|
\ll 
\left(\prod \limits_{i=1}^7|c_{i}|\right)^{\frac{14}{3}}
\ll M^{\frac{98}{3}}.$
\end{proof}

From now on, we assume that no ratio  $r_{i}$ is repeated more than six times, and  we provide an appropriate rearrangement of the variables on this basis.

\begin{lemma}
\label{A_B}
Suppose that among the ratios $r_{1},...,r_{16}$ no ratio is repeated more than six times. Then the indices 1,...16 can be rearranged into two disjoint sets $\mathscr{A}, \mathscr{B}$ with ten and six elements respectively and such that
\begin{itemize}
  \item [(i)]
  among the ratios $r_{j}$ with $j \in \mathscr{A}$ every ratio is repeated $\leq 4$ times,        
  \item [(ii)]
  the ratios $r_{j}$ with $j \in \mathscr{B}$ take on $\geq 3$ distinct values and every value occurs $\leq 2$ times.
\end{itemize}
\end{lemma}

This is Lemma 1 of \cite{Vaughan_Additive_1975}.

After the rearrangement, the indices in $\mathscr{A}$ are changed to 1,...,10. Those in $\mathscr{B}$ are changed to 11,...,16, and we may represent them using sequences of symbols selected from the set $\{r_a, r_b, \ldots, r_f\}$. There are four possible sequences, enumerated as follows:
\begin{itemize}
    \item[] $r_a, r_b, r_a, r_b, r_c, r_d$;
    \item[] $r_a, r_b, r_a, r_b, r_c, r_c$;
    \item[] $r_a, r_b, r_c, r_d, r_e, r_e$;
    \item[] $r_a, r_b, r_c, r_d, r_e, r_f$,
\end{itemize}
where the six symbols in each sequence denote the ratios of $\mathscr{B}$ in their respective positions, and identical symbols correspond to identical ratios.

\section{The circle method}

In this section, we employ the Hardy--Littlewood circle method to handle our problem. The underlying idea in the proof of Theorem \ref{zhudingli} is standard: 
we aim to establish the counting function $N(P)$ for solutions of \eqref{eq:zuikaishi} over a suitable bounded region $\mathscr{R}$.
To this end, we first choose a vector $\boldsymbol{\eta}$ based on the coefficients of \eqref{eq:zuikaishi} to determine the ranges of the first twelve coordinates in the definition \eqref{refion_R_definition} of $\mathscr{R}$.

\begin{lemma}
\label{lemma3.1}
Since the ratios $r_{1}, \ldots,r_{12}$ are not all equal in the cases listed above, we can choose a non-zero real vector $\boldsymbol{\eta}=(\eta_{1},\ldots,\eta_{12})$ such that
\begin{equation*}
     \begin{cases}
         c_{1}\eta_1+c_{2}\eta_2+\cdots+c_{12}\eta_{12}=0 ,\\
         d_{1}\eta_1+d_{2}\eta_2+\cdots+d_{12}\eta_{12}=0, 
     \end{cases}   
\end{equation*}
with $\eta_i > 0$ for each $i$, and $0<|\boldsymbol{\eta}|\ll M^2$.
More specifically,
$$
\eta_i \asymp 1, \quad \text{for } i=1,\ldots,10,
$$
$$
\frac{1}{M^2} \ll \eta_i \ll M^2, \quad \text{for }i=11,12.
$$
\end{lemma}

\begin{proof}
The simultaneous equations 
\begin{equation}
\label{eq_3_1_1} 
     \begin{cases}
    c_{1}x_1+c_{2}x_2+\cdots+c_{12}x_{12}=0 ,\\
    d_{1}x_1+d_{2}x_2+\cdots+d_{12}x_{12}=0 
     \end{cases}
\end{equation}
 always admit a solution for $x_{11}, x_{12}$ regardless of the values of $x_1, \ldots, x_{10}$, since the determinant of the matrix
\[
\begin{vmatrix} c_{11} & c_{12} \\ d_{11} & d_{12} \end{vmatrix} \neq 0
\]
in any case. 

We first consider the assignment that $x_1=...=x_{10}=1$. We can solve for $x_{11}, x_{12}$ from the equation 
$$\left ( \begin{matrix}
c_{11}& c_{12} \\
d_{11}& d_{12} \\
\end{matrix} \right) 
\left ( \begin{matrix}
x_{11} \\
x_{12} \\
\end{matrix} \right)
=
\left ( \begin{matrix}
n_{1} \\
n_{2} \\
\end{matrix} \right),$$
where 
$$
n_{1}=-(c_{1}x_1+c_{2}x_2+\cdots+c_{10}x_{10}),
$$ 
$$
n_{2}=-(d_{1}x_1+d_{2}x_2+\cdots+d_{10}x_{10}).
$$
There may be undesirable situations, such as either
$\left ( \begin{matrix}
n_{1} \\
n_{2} \\
\end{matrix} \right)
=
\left ( \begin{matrix}
0 \\
0 \\
\end{matrix} \right)$ 
or 
$\left ( \begin{matrix}
n_{1} \\
n_{2} \\
\end{matrix} \right)$ is parallel to one of
$\left ( \begin{matrix}
c_{11} \\
d_{11} \\
\end{matrix} \right)
$
and
$\left ( \begin{matrix}
c_{12} \\
d_{12} \\
\end{matrix} \right)
$.

To address this, we can always find a suitable variable among \( x_1, \ldots, x_{10} \) and adjust it from 1 to 2. Such slight adjustments keep \( x_i \asymp 1 \) for \( i = 1, \ldots, 10 \), and the adjusted vector, still denoted 
$\left ( \begin{matrix}
n_{1} \\
n_{2} \\
\end{matrix} \right)$, is linearly independent of both
$\left ( \begin{matrix}
c_{11} \\
d_{11} \\
\end{matrix} \right)
$
and
$\left ( \begin{matrix}
c_{12} \\
d_{12} \\
\end{matrix} \right)
$.
This is because Lemma \ref{A_B} ensures that no ratio in \( \mathscr{A} \) is repeated more than four times; in other words, there are at least three distinct ratios in \( \mathscr{A} \).

To solve the linear simultaneous equations above, we have
$$
x_{11}=\frac{\left | \begin{matrix}
n_{1}& c_{12} \\
n_{2}& d_{12} \\
\end{matrix} \right |}{\left | \begin{matrix}
c_{11}& c_{12} \\
d_{11}& d_{12} \\
\end{matrix} \right |}, \quad
x_{12}=\frac{\left | \begin{matrix}
c_{11}& n_{1} \\
d_{11}& n_{2}
\end{matrix} \right |}{\left | \begin{matrix}
c_{11}& c_{12} \\
d_{11}& d_{12} \\
\end{matrix} \right |},
$$
which satisfy 
$$
\frac{1}{M^2} \ll x_i \ll M^2\quad \text{for} \quad i=11,12. 
$$
Moreover, we can assume without loss of generality that $\eta_i > 0$ for all $i = 1, \dots, 12$, since whenever necessary the $c_i, d_i$ can be replaced by $-c_i, -d_i$, and $x_{i}^3$ by $-x_{i}^3$.

\end{proof}

Let
\begin{equation*}
\xi_{i}=\frac{1}{2}\eta_{i}^{1/3}, \quad \zeta_{i}=2\eta_{i}^{1/3} \quad (i=1,\ldots,12),
\end{equation*}
and
\begin{equation*}
T_{i}(\gamma)=\sum_{\xi_i P<x<\zeta_i P} e(\gamma x^3),
\end{equation*}
where $P$ is large in terms of $\varepsilon, M$.
We set $T_{i}(\gamma),\ (i=1,\ldots,12)$ in this way in order to apply Fourier's integral formula in Lemma \ref{I(P)_lower_bound} to obtain a lower bound for the singular integral.

We further define
\begin{equation*}
    U(\gamma)=\sum_{P^{4/5}<x<2P^{4/5}} e(\gamma x^3),
\end{equation*}
where we define \(U(\gamma)\) in this way because we can invoke Lemma 10 of Davenport \cite{Davenport1939} in our Lemma \ref{lemma5.6} to obtain an upper bound that is better than that obtained using Weyl's inequality.

The region $\mathscr{R}$ under consideration is determined by
\begin{equation}
\label{refion_R_definition}
    \begin{cases}
        \xi_i P<x_i<\zeta_i P \quad (i=1,\ldots,12),\\
        P^{4/5}<x_i<2P^{4/5} \quad (i=13,\ldots,16).
    \end{cases}
\end{equation}

Let $\alpha_1,\alpha_2$ be real variables,
\begin{equation*}
     \gamma_i=c_i\alpha_1+d_i\alpha_2 \quad (i=1,\ldots ,16).
\end{equation*}

$\delta$ is a parameter to be determined,
and
\begin{equation}
    \label{eta_definition}
     \eta=P^{-2-\delta}.
\end{equation}
Then
\begin{equation}
\label{N(P)_definition}
\begin{split}
    N(P)
    &=\# \{\mathbf{x} \in \mathscr{R}: F(\mathbf{x})= 0\text{  and }G(\mathbf{x})= 0\}\\
    &=\int_{\eta}^{1+\eta} \int_{\eta}^{1+\eta} T_{1}(\gamma_1)...T_{12}(\gamma_{12})U(\gamma_{13})...U(\gamma_{16})  d\alpha_1d\alpha_2.
\end{split}
\end{equation}
The open square $(\eta,1+\eta)\times(\eta,1+\eta)$ is dissected in the following way.
We denote a typical major arc by
\begin{equation}
\label{major_arc_definition}
     \mathfrak{M}(a_1,a_2,q)=\{(\alpha_1,\alpha_2):|q\alpha_i-a_i|<P^{-2-\delta} (i=1,2)\},
\end{equation}
where
\begin{equation*}
     (a_1,a_2,q)=1\quad \text{and}\quad 1 \leq a_1, a_2 \leq q \leq P^{1-\delta}.
\end{equation*}
The major arcs $\mathfrak{M}(a_1,a_2,q)$ are disjoint since, whenever $a/q \neq a'/q'$ and $q,q' \leq P^{1-\delta}$, 
$$
\left|a/q-a'/q'\right| \geq 1/(qq') > (1/q+1/q')P^{-2-\delta}.
$$
Let $\mathfrak{M}$ denote the union of the major arcs, and $\mathfrak{m}$ the minor arcs,
\begin{equation*}
     \mathfrak{m}=(\eta,1+\eta)\times(\eta,1+\eta) \setminus \mathfrak{M}.
\end{equation*}

\section{The minor arcs}

We only present the necessary details of the derivation, with most parts briefly sketched, as this section is essentially the effective version of Section 4 in Vaughan \cite{Vaughan_Additive_1975}.

\begin{lemma}
\label{lemma4.1}
We have
    \begin{equation}
    \label{jifen_2_2_2_2}
        \int_{0}^{1} \int_{0}^{1} |T_{11}(\gamma_{11})T_{12}(\gamma_{12})U(\gamma_{13})...U(\gamma_{16})|^2  d\alpha_1d\alpha_2 \ll P^{26/5+\epsilon}M^{34/3+\epsilon}.
    \end{equation}

\end{lemma}

    This is essentially \cite[Lemma 19]{Davenport1966CubicEO}. We can suppose without loss of generality that $r_{15}=r_{16}$, because by the Cauchy--Schwarz inequality,
    $$ |U(\gamma_{15})U(\gamma_{16})|^2 \leq |U(\gamma_{15})|^4 + |U(\gamma_{16})|^4. 
    $$

Consequently, there are two cases for $r_{11},...,r_{16}$, represented by the symbol sequences as before,
\begin{itemize}
  \item [(i)]
  $r_a, r_b, r_a, r_b, r_c, r_c,$
  \item [(ii)]
  $r_a, r_b, r_c, r_d, r_e, r_e.$

\end{itemize}
In either case, the integral in \eqref{jifen_2_2_2_2} represents the number of solutions of the simultaneous equations
\begin{equation}
    \label{lemma_4_1_2}
    \begin{cases}
        c_{11}x_{11}^3+...+c_{16}x_{16}^3=c_{11}y_{11}^3+...+c_{16}y_{16}^3,\\
        d_{11}x_{11}^3+...+d_{16}x_{16}^3=d_{11}y_{11}^3+...+d_{16}y_{16}^3,
    \end{cases}
\end{equation}
where the variables are integers subject to
\begin{equation}
    \label{case_i_variable_range}
    \begin{cases}
        \xi_i P<x_i,y_i<\zeta_i P (i=11,12),\\
        P^{4/5}<x_i,y_i<2P^{4/5} (i=13,...,16).
    \end{cases}
\end{equation}

In case (i), we can form linear combinations of the two equations in \eqref{lemma_4_1_2} to eliminate $x_{11}$ (and therefore also $x_{13}$) or to eliminate $x_{12}$ (and therefore also $x_{14}$).
Recalling that $r_{15}=r_{16}=r_{c}$, and relettering the remaining variables, we obtain two equations of the form
\begin{equation}
    \label{case_i's equations}
    \begin{cases}
        (c_{11}-r_{b} d_{11})x_{11}^3+(c_{13}-r_{b} d_{13})x_{13}^3+d_{15}(r_{c}-r_{b})x_{15}^3+d_{16}(r_{c}-r_{b})x_{16}^3\\
        =(c_{11}-r_{b} d_{11})y_{11}^3+(c_{13}-r_{b} d_{13})y_{13}^3+d_{15}(r_{c}-r_{b})y_{15}^3+d_{16}(r_{c}-r_{b})y_{16}^3,\\
                \\
        (c_{12}-r_{a} d_{12})x_{12}^3+(c_{14}-r_{a} d_{14})x_{14}^3+d_{15}(r_{c}-r_{a})x_{15}^3+d_{16}(r_{c}-r_{a})x_{16}^3\\
        =(c_{12}-r_{a} d_{12})y_{12}^3+(c_{14}-r_{a} d_{14})y_{14}^3+d_{15}(r_{c}-r_{a})y_{15}^3+d_{16}(r_{c}-r_{a})y_{16}^3,\\
        
    \end{cases}
\end{equation}
subject to \eqref{case_i_variable_range}.
Under the condition of case (i), none of the new coefficients is 0.
We shall investigate the number of solutions of \eqref{case_i's equations} subject to \eqref{case_i_variable_range} in Lemmas \ref{lemma 4.2} -- \ref{lemma 4.4}.

In case (ii), we can again form linear combinations of the two equations in \eqref{lemma_4_1_2} to eliminate $x_{11}$ or $x_{12}$, but now none of the other variables disappears. We obtain the two equations of the form
\begin{equation}
    \label{case_ii's equations}
    \begin{cases}
        (c_{11}-r_{b} d_{11})x_{11}^3+
        (c_{13}-r_{b} d_{13})x_{13}^3+(c_{14}-r_{b} d_{14} )x_{14}^3\\
        \qquad \qquad+d_{15}(r_{e}-r_{b} )x_{15}^3+d_{16}(r_{e}-r_{b} )x_{16}^3\\
        =(c_{11}-r_{b} d_{11})y_{11}^3+(c_{13}-r_{b} d_{13})y_{13}^3+(c_{14}-r_{b} d_{14} )y_{14}^3\\
        \qquad \qquad+d_{15}(r_{e}-r_{b} )y_{15}^3+d_{16}(r_{e}-r_{b} )y_{16}^3,\\
    
                \\
        
        (c_{12}-r_{a} d_{12})x_{12}^3+(c_{13}-r_{a} d_{13})x_{13}^3+(c_{14}-r_{a} d_{14})x_{14}^3\\
        \qquad \qquad+d_{15}(r_{e}-r_{a})x_{15}^3+d_{16}(r_{e}-r_{a})x_{16}^3\\
        =(c_{12}-r_{a} d_{12})y_{12}^3+(c_{13}-r_{a} d_{13})y_{13}^3+(c_{14}-r_{a} d_{14})y_{14}^3\\
        \qquad \qquad+d_{15}(r_{e}-r_{a})y_{15}^3+d_{16}(r_{e}-r_{a})y_{16}^3,\\
        
    \end{cases}
\end{equation}
again subject to \eqref{case_i_variable_range}.

Under the condition of case (ii), none of the new coefficients is 0, and the ratios
$$
\frac{(c_{13}-r_{b} d_{13})}{(c_{13}-r_{a} d_{13})},\quad \frac{(c_{14}-r_{b} d_{14})}{(c_{14}-r_{a} d_{14})},\quad \frac{d_{15}(r_{e}-r_{b})}{d_{15}(r_{e}-r_{a})}
$$
are distinct. We shall investigate the number of solutions of \eqref{case_ii's equations} subject to 
\eqref{case_i_variable_range} in Lemmas \ref{lemma 4.5} -- \ref{lemma 4.7}.

\begin{lemma}
\label{lemma 4.2}
The number of solutions of \eqref{case_i's equations} subject to \eqref{case_i_variable_range}, with $x_{11}=y_{11}$ and $x_{12}=y_{12}$ is $\ll P^{26/5+\epsilon}M^{4/3+\epsilon}$. 
\end{lemma}

\begin{proof}
This is essentially Lemma 20 of Davenport and Lewis \cite{Davenport1966CubicEO}, and the tedious computational details are omitted here. 
\end{proof}

\begin{lemma}
\label{lemma 4.3}
The number of solutions of \eqref{case_i's equations} subject to \eqref{case_i_variable_range}, with $x_{11}=y_{11}$ and $x_{12} \neq y_{12}$ is $\ll P^{26/5+\epsilon}M^{14/3+\epsilon}$.
\end{lemma}

\begin{proof}
This is essentially Lemma 21 of Davenport and Lewis \cite{Davenport1966CubicEO}, and the tedious computational details are omitted here. 
\end{proof}

\begin{lemma}
\label{lemma 4.4}
The number of solutions of \eqref{case_i's equations} subject to \eqref{case_i_variable_range}, with $x_{11} \neq y_{11}$ and $x_{12} \neq y_{12}$ is $\ll P^{26/5+\epsilon}M^{34/3+\epsilon}$.
\end{lemma}

\begin{proof}
This is essentially Lemma 22 of Davenport and Lewis \cite{Davenport1966CubicEO}, and the tedious computational details are omitted here. 
\end{proof}

\begin{lemma}
\label{lemma 4.5}
The number of solutions of \eqref{case_ii's equations} subject to \eqref{case_i_variable_range},
with $x_{11}=y_{11}$ and $x_{12}=y_{12}$ is $\ll P^{26/5+\epsilon}M^{4/3+\epsilon}$. 
\end{lemma}

\begin{proof}
This is essentially Lemma 23 of Davenport and Lewis \cite{Davenport1966CubicEO}, and the tedious computational details are omitted here. 
\end{proof}

\begin{lemma}
\label{lemma 4.6}
The number of solutions of \eqref{case_ii's equations} subject to \eqref{case_i_variable_range},
with $x_{11}=y_{11}$ and $x_{12} \neq y_{12}$ is $\ll P^{5+\epsilon}M^{2/3+\epsilon}$.
\end{lemma}

\begin{proof}
This is essentially Lemma 24 of Davenport and Lewis \cite{Davenport1966CubicEO}, and the tedious computational details are omitted here. 
\end{proof}

\begin{lemma}
\label{lemma 4.7}
The number of solutions of \eqref{case_ii's equations} subject to \eqref{case_i_variable_range},
with $x_{11} \neq y_{11}$ and $x_{12} \neq y_{12}$ is $\ll P^{26/5+\epsilon}M^{\epsilon}$.
\end{lemma}

\begin{proof}
This is essentially Lemma 25 of Davenport and Lewis \cite{Davenport1966CubicEO}, and the tedious computational details are omitted here. 
\end{proof}

\begin{proof}[Proof of Lemma \ref{lemma4.1}]
This follows, by virtue of the preliminary remarks, from lemmas \ref{lemma 4.2} to \ref{lemma 4.4} in case (i) and from lemmas \ref{lemma 4.5} to \ref{lemma 4.7} in case (ii). 
\end{proof}

\begin{lemma}
\label{jifen_8_8}
Suppose that $1 \leq i,j \leq 10$ and $r_i \neq r_j$. Then
$$
\int_{0}^{1} \int_{0}^{1} |T_{i}(\gamma_i)T_{j}(\gamma_{j})|^8  d\alpha_1d\alpha_2 \ll P^{10+\epsilon}M^{4}.
$$
\end{lemma}

\begin{proof}
The argument is similar to Lemma 2 of Cook \cite{Cook1972PairsOA}. For $0< \alpha_1 <1$ and $0< \alpha_2 <1$, recalling
\begin{equation*}
     \gamma_i=c_i\alpha_1+d_i\alpha_2 \quad (i=1,\ldots ,16),
\end{equation*}
we  have the estimates
$$
\max\{|\gamma_i|,|\gamma_j|\} \ll M,
$$
and
$$
\Delta =|c_id_j-c_j d_i| =\left|\frac{\partial (\gamma_i,\gamma_j)}{\partial (\alpha_1,\alpha_2)}\right|\ll M^2.
$$

Since \( \Delta \neq 0 \), we can change the variables of integration from \( \alpha_1, \alpha_2 \) to \( \gamma_i, \gamma_j \). Using the periodicity of the integrand and applying Hua's Inequality \cite[Lemma 3.2]{davenport_browning_2005}, we obtain the estimates
\begin{equation*}
\begin{split}
    \int_{0}^{1} \int_{0}^{1} |T_i(\gamma_i) T_j(\gamma_j)|^8 \, d\alpha_1 \, d\alpha_2 
    &\ll M^2 \int_{0}^{M} \int_{0}^{M} |T_i(\gamma_i) T_j(\gamma_j)|^8 \, d\gamma_i \, d\gamma_j \\
    &\ll M^4 \int_{0}^{1} \int_{0}^{1} |T_i(\gamma_i) T_j(\gamma_j)|^8 \, d\gamma_i \, d\gamma_j \\
    &\ll M^4 \left( \left( \int_{0}^{1} |T_i(\gamma_i)|^8 \, d\gamma_i \right)^2 + \left( \int_{0}^{1} |T_j(\gamma_j)|^8 \, d\gamma_j \right)^2 \right) \\
    &\ll P^{10+\varepsilon} M^4.
\end{split}
\end{equation*}

\end{proof}

Let 
\begin{equation}
\label{S(a,q)_definition}
    S(a,q)=\sum_{m=1}^{q} e\left(\frac{am^3}{q}\right).
\end{equation}

\begin{lemma}

    Suppose that $(a,q)=1.$ Then
    \begin{equation}
    \label{S(a,q)_upper_bound}
    S(a,q)\ll q^{\frac{2}{3}}.
    \end{equation}
 
\end{lemma}
This is Lemma 3 of Hardy and Littlewood \cite{Hardy1925SomePO}.

\begin{lemma}
    \label{lemma4.10}
    Suppose that $q \leq P^{1-\delta}, (a,q)=1,$ and $|\gamma q-a| < P^{-2-\delta}.$ Then, for $i=1,...,10,$
    \begin{equation}
    \label{lemma4.10_inequa_1}
    T_{i}(\gamma)\ll \frac{P|S(a,q)|}{q(1+P^3|\gamma-a/q|)}+q^{2/3+\epsilon},
    \end{equation}
and 
    \begin{equation}
    \label{lemma4.10_inequa_2}
    T_{i}(\gamma)\ll \frac{P}{q^{1/3}(1+P^3|\gamma-a/q|)}.
    \end{equation}
\end{lemma}

This is Lemma 5 of R. C. Vaughan \cite{Vaughan_Additive_1975}.

\begin{lemma}
    \label{lemma4.11}
    Suppose that $1\leq i,j,k \leq 10$, and $r_i, r_j, r_k$ are distinct. Then
    \begin{equation*}
        \iint_{\mathfrak{m}}|{T_{i}(\gamma_i)}^8{T_{j}(\gamma_{j})}^8{T_{k}(\gamma_{k})}^4|  d\alpha_1d\alpha_2 \ll P^{13+4\delta+\epsilon}M^{4}.
    \end{equation*}
\end{lemma}

\begin{proof}
Let
\begin{equation*}
    \mathfrak{M}_1(b_i,b_j,q_i,q_j)=\left\{(\alpha_1,\alpha_2):|q_r\gamma_r-b_r| < \frac{1}{P^{2+\delta}}, (r=i, j)\right\},
\end{equation*}
where
\begin{equation}
    \label{Majar_arcs_1_requirment_1}
    q_i, q_j \leq P^{1-\delta},\quad (b_i, q_i)=(b_j, q_j)=1,
\end{equation}
and
\begin{equation}
    \label{Majar_arcs_1_requirment_2}
    |b_r| \leq 2(|c_r|+|d_r|)q_r,\quad (r=i, j).
\end{equation}

The sets \( \mathfrak{M}_1(b_i, b_j, q_i, q_j) \) are clearly disjoint. We define
\[
\mathfrak{M}_1 = \bigcup_{\substack{b_i, b_j, q_i, q_j \\ q_i q_j > P^{3/4}}} \mathfrak{M}_1(b_i, b_j, q_i, q_j), \quad \mathfrak{m}_1 = \mathfrak{m} \setminus \mathfrak{M}_1.
\]
Note that \( \mathfrak{m} \setminus \mathfrak{m}_1 \subsetneq \mathfrak{M}_1 \). So \( \mathfrak{m} \subseteq   \mathfrak{m}_1 \cup \mathfrak{M}_1 \). Next we estimate the integrals over $\mathfrak{m}_1 \text{ and } \mathfrak{M}_1$, separately.

Step 1. We first handle $\mathfrak{m_1}$. Let $(\alpha_1,\alpha_2) \in \mathfrak{m_1}$. By Dirichlet's theorem, we may choose $b_i,b_j,q_i,q_j$ so that
\begin{equation}
    \label{lemma_4.11_3}
    |q_r\gamma_r-b_r| \leq \frac{1}{P^{2+\delta}},\quad(q_r,b_r)=1,\quad q_r \leq P^{2+\delta}\quad (r=i,j).
\end{equation}
Recalling that $\eta=P^{-2-\delta}$ is given by \eqref{eta_definition}, from \eqref{lemma_4.11_3} we know
\begin{equation*}
    |b_r|\leq |q_r\gamma_r|+\eta \leq 1.5(|c_r|+|d_r|)q_r +\eta \leq 2(|c_r|+|d_r|)q_r
\end{equation*}
holds, which is \eqref{Majar_arcs_1_requirment_2}.

Under the condition $(\alpha_1,\alpha_2) \in \mathfrak{m_1}$, we proceed to discuss the range of $q_i$ through classification. The first possible case is 
\begin{equation}
    \label{32}
    q_r>P^{1-\delta}\quad (r=i \text{ or } j).
\end{equation}
By Weyl's inequality \cite[Lemma 2.3] {Hardy1925SomePO},
\begin{equation*}
    T_{r}(\gamma_{r}) \ll P^{1+\epsilon}\left(\frac{1}{q_r}+\frac{1}{P}+\frac{q_r}{P^3}\right)^{\frac{1}{4}},
\end{equation*}
we obtain
\begin{equation}
    \label{T_r_upper_bound_in_m_1}
    T_{r}(\gamma_{r}) \ll P^{\frac{3}{4}+\delta}.
\end{equation}
The second possible case is for $r=i \text{ or } j$
\begin{equation*}
        q_r \leq P^{1-\delta}\text{ and }P^{\frac{1}{4}-\delta}<q_r^{\frac{1}{3}}\left(1+P^3\left|\gamma_r-\frac{b_r}{q_r}\right| \right).
\end{equation*}
By \eqref{lemma4.10_inequa_2} we obtain \eqref{T_r_upper_bound_in_m_1} again.
 There are no other possibilities, which can be demonstrated by the proof of Lemma 6 of Vaughan \cite{Vaughan_Additive_1975}. 
 
 Thus, by \eqref{T_r_upper_bound_in_m_1}, for every $(\alpha_1,\alpha_2)\in \mathfrak{m}_1$,
 we know
\begin{equation*}
    \min \{|T_{i}(\gamma_{i})|,|T_{j}(\gamma_{j})|\} \ll P^{\frac{3}{4}+\delta}.
\end{equation*}
Hence
\begin{equation*}
    \begin{split}
    &\iint_{\mathfrak{m_1}}|{T_{i}(\gamma_i)}^8{T_{j}(\gamma_{j})}^8{T_{k}(\gamma_{k})}^4|  d\alpha_1d\alpha_2 \\
    &\ll P^{3+4\delta}\int_{0}^{1}\int_{0}^{1} {|T_{i}(\gamma_{i})T_{j}(\gamma_{j})|}^4({|T_{i}(\gamma_{i})|}^4+|T_{j}(\gamma_{j})|^4) |T_{k}(\gamma_{k})|^4 d\alpha_1d\alpha_2 .
    \end{split}
\end{equation*}
Therefore, by the Cauchy--Schwarz inequality and Lemma \ref{jifen_8_8},
\begin{equation}
    \label{minor_1_8_8_4_upper_bound}
    \iint_{\mathfrak{m_1}}|{T_{i}(\gamma_i)}^8{T_{j}(\gamma_{j})}^8{T_{k}(\gamma_{k})}^4|  d\alpha_1d\alpha_2 \ll P^{13+4\delta+\epsilon}M^{4}.
\end{equation}

Step 2. It remains to treat $\mathfrak{M_1}$.
First, by the Cauchy--Schwarz inequality and lemma \ref{jifen_8_8},
\begin{equation}
\label{4_4_8_upper_bound}
\begin{split}
&\int_{0}^{1}\int_{0}^{1} |{T_{i}(\gamma_i)}^4{T_{j}(\gamma_{j})}^4{T_{k}(\gamma_{k})}^8|  d\alpha_1d\alpha_2\\
&=\int_{0}^{1}\int_{0}^{1} |{T_{i}(\gamma_i)}^4{T_{k}(\gamma_{k})}^4||{T_{j}(\gamma_j)}^4{T_{k}(\gamma_{k})}^4|  d\alpha_1d\alpha_2\\
&\ll \left(\int_{0}^{1}\int_{0}^{1} |{T_{i}(\gamma_i)}^8{T_{k}(\gamma_{k})}^8|  d\alpha_1d\alpha_2\right)^{1/2}\left(\int_{0}^{1}\int_{0}^{1} |{T_{j}(\gamma_j)}^8{T_{k}(\gamma_{k})}^8|  d\alpha_1d\alpha_2\right)^{1/2}\\
&\ll P^{10+\epsilon}M^{4}.\\
\end{split}
\end{equation}
Hence, invoking Cauchy-Schwarz inequality again and by \eqref{4_4_8_upper_bound}, we have
\begin{equation}
    \label{M_1_8_8_4_upper_bound}
    \begin{split}
    &\iint_{\mathfrak{M_1}}|{T_i(\gamma_i)}^8{T_j(\gamma_j)}^8{T_k(\gamma_k)}^4|  d\alpha_1d\alpha_2 \\
    &=\iint_{\mathfrak{M_1}}|{T_i(\gamma_i)}^2{T_j(\gamma_j)}^2{T_k(\gamma_k)}^4|\cdot |{T_i(\gamma_i)}^6{T_j(\gamma_j)}^6|  d\alpha_1d\alpha_2\\
    &\ll \left(\int_{0}^{1}\int_{0}^{1} |{T_i(\gamma_i)}^4{T_j(\gamma_j)}^4{T_k(\gamma_k)}^8|  d\alpha_1d\alpha_2\right)^{1/2} \left(\iint_{\mathfrak{M_1}} {|T_i(\gamma_i)T_j(\gamma_j)|}^{12}  d\alpha_1d\alpha_2\right)^{1/2}\\
    &\ll P^{5+\epsilon}M^{2}\left(\iint_{\mathfrak{M_1}} {|T_i(\gamma_i)T_j(\gamma_j)|}^{12}  d\alpha_1d\alpha_2\right)^{1/2}.\\
    \end{split}
\end{equation}
By (4.15) to (4.17) of Vaughan \cite{Vaughan_Additive_1975}, we have
\begin{equation}
    \label{M_1(bi,bj,qi,qj)_upper_bound}
    \begin{split}
        &\iint_{\mathfrak{M_1}(b_i,b_j,q_i,q_j)} |T_i(\gamma_i)T_j(\gamma_j)|^{12}  d\alpha_1d\alpha_2\\
        &\ll P^{18}\frac{|S(b_i,q_i)S(b_j,q_j)|^{12}}{(q_iq_j)^{12}}+P^{14}\left(\frac{1}{q_i^4}+\frac{1}{q_j^4}\right).
    \end{split}
\end{equation}
We recall that
$$
\mathfrak{M}_1 = \bigcup_{\substack{b_i, b_j, q_i, q_j \\ q_i q_j > P^{3/4}}} \mathfrak{M}_1(b_i, b_j, q_i, q_j).
$$
Hence, by \eqref{Majar_arcs_1_requirment_1}, \eqref{Majar_arcs_1_requirment_2}, \eqref{M_1(bi,bj,qi,qj)_upper_bound}, and the periodicity of $S(a,q)$ we have 
\begin{equation}
\label{M_1_12_12_upper_bound}
    \begin{split}
        &\iint_{\mathfrak{M_1}} |T_i(\gamma_i)T_j(\gamma_j)|^{12}  d\alpha_1d\alpha_2\\
            &\ll P^{16}M^2\left(\sum_{q<P} q^{-28/3} \sum_{a=1}^{q}{}^{*} |S(a,q)|^{12}\right)^2,
    \end{split}
\end{equation}
where we note that
the second term in \eqref{M_1(bi,bj,qi,qj)_upper_bound}, when summed over all $b_i,b_j,q_i,q_j$, is absorbed in the term $q=1$.

Vaughan \cite[p. 360]{Vaughan_Additive_1975} proves
\begin{equation}
\label{Vaughan's_p_360}
    \begin{split}
        &\sum_{q<P} q^{-28/3} \sum_{a=1}^q{}^{*} |S(a,q)|^{12} \leq \prod \limits_{p\leq P} \left(1+\frac{C}{p}\right) \text{ for some constant } C>0.
    \end{split}
\end{equation}
By \eqref{M_1_12_12_upper_bound} and \eqref{Vaughan's_p_360}, and noting that
\begin{equation*}
    \begin{split}
        \prod \limits_{p\leq P} \left(1+\frac{C}{p}\right) \ll P^{\epsilon},
    \end{split}
\end{equation*}
we have
$$        
\iint_{\mathfrak{M_1}} |T_i(\gamma_i)T_j(\gamma_j)|^{12}  d\alpha_1d\alpha_2 \ll P^{16+\epsilon}M^2.
$$
Therefore, by \eqref{M_1_8_8_4_upper_bound},
$$
\iint_{\mathfrak{M_1}}|{T_i(\gamma_i)}^8{T_j(\gamma_j)}^8{T_k(\gamma_k)}^4|  d\alpha_1d\alpha_2 \ll P^{13+\epsilon}M^3.
$$
This with \eqref{minor_1_8_8_4_upper_bound} completes the proof of the lemma.
\end{proof}

\begin{lemma}
\label{lemma4.12}
    On the hypothesis of Lemma \ref{lemma4.11}, we have
    \begin{equation*}
        \iint_{\mathfrak{m}}|{T_{i}(\gamma_i)}^8{T_{j}(\gamma_{j})}^6{T_{k}(\gamma_{k})}^6|  d\alpha_1d\alpha_2 \ll P^{13+4\delta+\epsilon}M^{4}.
    \end{equation*}
\end{lemma}

\begin{proof}
    The proof is immediate from Lemma \ref{lemma4.11} and the trivial inequality
    \begin{equation*}
        |T_jT_k|^2 \leq |T_j|^4 + |T_k|^4.
    \end{equation*}
\end{proof}

\begin{lemma}
\label{lemma4.13}
    We have
    \begin{equation*}
        \iint_{\mathfrak{m}}|{T_{1}(\gamma_1)}...{T_{10}(\gamma_{10})}|^2  d\alpha_1d\alpha_2 \ll P^{13+4\delta+\epsilon}M^{4}.
    \end{equation*}
\end{lemma}

\begin{proof}
    Since the ratios $r_1,...,r_{10}$ arise from the set $\mathscr{A}$ of Lemma \ref{A_B}, no ratio is repeated more than four times and there are at least three distinct ratios. Thus, by several applications of the inequality
    \begin{equation*}
       |z_1...z_m| \leq |z_1|^m+...+|z_m|^m,
    \end{equation*}
    we can always reduce to Lemma \ref{lemma4.11} or Lemma \ref{lemma4.12}.
    This gives the lemma.
    
\end{proof}

\begin{lemma}
\label{lemma4.14}
    We have
    \begin{equation*}
        \iint_{\mathfrak{m}}|{T_{1}(\gamma_1)}...{T_{12}(\gamma_{12})}U(\gamma_{13})...U(\gamma_{16})|  d\alpha_1d\alpha_2 \ll P^{\frac{91}{10}+2\delta+\epsilon}M^{\frac{23}{3}+\epsilon}.
    \end{equation*}
\end{lemma}

\begin{proof}
    The proof is immediate from Lemma \ref{lemma4.1} and \ref{lemma4.13} and Cauchy–Schwarz inequality.
\end{proof}

\section{The major arcs}

In this section, we discuss the contribution from the major arcs. 

In fact, we will further truncate the major arcs, introduce an upper bound \( P_0 \) to restrict the denominators of the major arcs, and further reduce the radius of the intervals to \( P^{\delta - 3} \).
We will see in Lemma \ref{lemma5.8} that this part yields the main term.

This section is essentially the effective version of Section 5 in Vaughan \cite{Vaughan_Additive_1975}. We have added more details of the derivation based on that work to illustrate the dependence of different parts on the magnitude of \( M \).

\begin{lemma}
    \label{lemma5.1}
    Suppose that $(\alpha_1,\alpha_2) \in \mathfrak{M}(a_1,a_2,q)$,
    \begin{equation}
    \label{q_i_definition}
        q_i=q_i(a_1,a_2,q)=\frac{q}{(q,c_ia_1+d_ia_2)},
    \end{equation}
\begin{equation}
    \label{b_i_definition}
    b_i=b_i(a_1,a_2,q)=\frac{c_ia_1+d_ia_2}{(q,c_ia_1+d_ia_2)},
\end{equation}
    \begin{equation}
     \label{beta_j_definition}
        \beta_j=\alpha_j-\frac{a_j}{q}  \quad (j=1,2),
    \end{equation}
    and
    \begin{equation}
     \label{mu_i_definition}
        \mu_i =c_i\beta_1+d_i\beta_2.
    \end{equation}
    Then
    \begin{equation}
     \label{lemma5.1_conclusion}
        T_i(\gamma_i) \ll
        \begin{cases}
        \frac{P}{q_i^{\frac{1}{3}}(1+P^3|\mu_i|)},\quad &i=1,...,10,\\
        \frac{PM^{\frac{8}{3}}}{q_i^{\frac{1}{3}}(1+P^3|\mu_i|)},\quad &i=11,12.\\
        \end{cases}
    \end{equation}
\end{lemma}

\begin{proof}
Note that
    $$
    \gamma_i=\frac{b_i}{q_i}+\mu_i.
    $$
    By Lemma 7.11 of Hua \cite{1965Additive}, 
    we have
    \begin{equation}
     \label{lemma5.1_5}
        T_i(\gamma_i)=\frac{S(b_i,q_i)}{q_i}\int_{\xi_i P}^{\zeta_i P} e(\mu_i y^3) dy+O(q_i^{2/3+\epsilon}).
    \end{equation}
    
    When $i=1,...,10$, we have 
    $$
    |\mu_i|\ll \frac{M}{qP^{2+\delta}}\ll \frac{1}{qP^{2+\frac{\delta}{2}}}
    $$
    from 
 \eqref{major_arc_definition}, \eqref{mu_i_definition}. So the front half of \eqref{lemma5.1_conclusion} is the consequence of Lemma \ref{lemma4.10}. 
    When i=11 or 12, we deduce from the Euler--Maclaurin summation formula that
   \begin{equation*}
        \begin{split}
            \int_{\xi_i P}^{\zeta_i P} e(\mu_i y^3) dy
            &=\int_{\frac{1}{8}\eta_i P^3}^{8\eta_i P^3} \frac{e(\mu_i t)}{3t^{\frac{2}{3}}} dt\\
            &=\sum_{\frac{1}{8}\eta_i P^3<n\leq 8\eta_i P^3} \frac{e(\mu_i n)}{3n^{\frac{2}{3}}}-\frac{1}{2}\left(\frac{e(\mu_i 8\eta_i P^3)}{3(8\eta_i P^3)^{\frac{2}{3}}}-\frac{e(\mu_i \frac{1}{8}\eta_i P^3)}{3(\frac{1}{8}\eta_i P^3)^{\frac{2}{3}}}\right)\\
            & -\int_{\frac{1}{8}\eta_i P^3}^{8\eta_i P^3} B_{1}(t)\frac{2e(\mu_i t)}{9}\frac{3\pi i \mu_i t-1}{t^{\frac{5}{3}}} dt\\
            &=\sum_{\frac{1}{8}\eta_i P^3<n\leq 8\eta_i P^3}  \frac{e(\mu_i n)}{3n^{\frac{2}{3}}}+O(M^{\frac{4}{3}}P^{-2})+O(|\mu_i|PM^{\frac{2}{3}}),
        \end{split}
   \end{equation*}
where $B_1(t)$ is the first Bernoulli function , i.e., $B_1(t)=t-\frac{1}{2}$ on [0,1), with period 1. 
Furthermore, by Abel transformation,
    \begin{equation*}
        \begin{split}
            \sum_{\frac{1}{8}\eta_i P^3<n\leq 8\eta_i P^3} \frac{e(\mu_i n)}{3n^{\frac{2}{3}}}
            &\ll \min \left\{\eta_iP^3,\frac{1}{|\mu_i|}\right\}\frac{1}{(\eta_iP^3)^{\frac{2}{3}}}\\
            &\ll \frac{(\eta_iP^3)^{\frac{1}{3}}}{1+(\eta_iP^3)|\mu_i|}\ll M^{\frac{8}{3}}\frac{P}{1+P^3|\mu_i|}. 
        \end{split}
    \end{equation*}
    So
    \begin{equation}
     \label{lemma5.1_6}
            \int_{\xi_i P}^{\zeta_i P} e(\mu_i y^3) dy
            \ll M^{\frac{8}{3}}\frac{P}{1+P^3|\mu_i|}.
    \end{equation}
    Then the second half of \eqref{lemma5.1_conclusion} follows from \eqref{S(a,q)_upper_bound}.
\end{proof}

\begin{lemma}
\label{lemma5.2}
    On the hypothesis of Lemma \ref{lemma5.1},
    \begin{equation*}
        U(\gamma_i)\ll \frac{P^{\frac{4}{5}+\delta}}{q_i^{\frac{1}{5}}}, \quad(i=13,...,16).
    \end{equation*}
\end{lemma}

\begin{proof}
    By Weyl's inequality,
    \begin{equation*}
    \begin{split}
        U(\gamma_i) 
        &\ll P^{\frac{4}{5}+\epsilon}\left(\frac{1}{q_i}+\frac{1}{P^{\frac{4}{5}}}+\frac{q_i}{P^{\frac{12}{5}}}\right)^{\frac{1}{4}}\\
        &\ll \frac{P^{4/5+\epsilon}}{q_i^{1/4}}+P^{3/5+\epsilon}\\
        &\ll \frac{P^{\frac{4}{5}+\delta}}{q_i^{\frac{1}{5}}}.
    \end{split}
    \end{equation*}
    
\end{proof}

\begin{lemma}
\label{lemma5.3}
    We have
    \begin{equation}
    \label{lemma5.3_conclusion_1}
        \sum_{a_1,a_2} \frac{1}{(q_1...q_{12})^{\frac{1}{3}}(q_{13}...q_{16})^{\frac{1}{5}}} \ll q^{-\frac{9}{5}+\epsilon}M^{168+\frac{4}{5}},
    \end{equation}
    and
    \begin{equation}
    \label{lemma5.3_conclusion_2}
        \sum_{a_1,a_2} \frac{1}{(q_1...q_{16})^{\frac{1}{3}}} \ll q^{-\frac{7}{3}+\epsilon}M^{169+\frac{1}{3}},
    \end{equation}
    where in each case the summation is over $1\leq a_1,a_2\leq q$ with $(a_1,a_2,q)=1$.
\end{lemma}

\begin{proof}
    We collect together the blocks of equal ratios among $\frac{c_i}{d_i}  (i=1,...,16)$, each block contains at most six, so there are at least three blocks. 
    We write $\frac{c_i}{d_i}=\frac{c_j}{d_j}=\frac{c}{d}$ for some integers $(c,d)=1$, more specifically,
    $$
    \frac{c_i}{c}=\frac{d_i}{d}=k_1,\quad \frac{c_j}{c}=\frac{d_j}{d}=k_2,
    $$
    where $k_1, k_2$ are two integers $\ll M$.
    By the definition \eqref{q_i_definition} of $q_i,$ we have
    \begin{equation*}
        \frac{q_i}{q_j}=\frac{(q,c_ja_1+d_ja_2)}{(q,c_ia_1+d_ia_2)}=\frac{(q,k_2(ca_1+da_2))}{(q,k_1(ca_1+da_2))},
    \end{equation*}    
    therefore,
    \begin{equation*}
        \frac{1}{M} \ll \frac{q_i}{q_j} \ll M.
    \end{equation*}
    Thus if $i_1,...,i_\nu$ is a representative set of indices, one from each block, then
    \begin{equation}
    \label{lemma5.3_proof_1}
        (q_1...q_{12})^{-\frac{1}{3}}(q_{13}...q_{16})^{-\frac{1}{5}} \ll   M^{\frac{24}{5}}q_{i_1}^{-\theta_1}...q_{i_\nu}^{-\theta_\nu},
    \end{equation}
    where 
    $$
    \frac{1}{5} \leq \theta_j \leq 2 \quad (j=1,...,\nu)
    $$
    and
    \begin{equation}
    \label{lemma5.3_proof_2}
        \theta_1+...+\theta_\nu=\frac{12}{3}+\frac{4}{5}=\frac{24}{5}.
    \end{equation}
    
    Let 
    $$
    u_j=(q,c_{i_j}a_1+d_{i_j}a_2)
    $$
    for $j=1,...,\nu,$ so that $u_j|q$ and 
    \begin{equation}
    \label{lemma5.3_proof_3}
        q_{i_j}=\frac{q}{u_j}
    \end{equation}
   by \eqref{q_i_definition}. We know $\frac{c_{i_j}}{d_{i_j}} \neq \frac{c_{i_k}}{d_{i_k}}$ if $j \neq k$. 
   
   Let $\delta=(u_j,u_k).$ Then
   $$ c_{i_j}a_1+d_{i_j}a_2 \equiv 0\ (\text{mod} \ \delta),$$
   $$ c_{i_k}a_1+d_{i_k}a_2 \equiv 0\ (\text{mod} \ \delta),$$
   whence
   $$ (c_{i_j}d_{i_k}-c_{i_k}d_{i_j})a_1 \equiv 0\ (\text{mod} \ \delta),$$
   and similar for $a_2$. Since $(a_1,a_2,q)=1$, and $\delta|q$, it follows that
   $$ \delta|c_{i_j}d_{i_k}-c_{i_k}d_{i_j}.$$
   Thus for any $j,k$ with $j \neq k$, we have $\delta=(u_j,u_k) \ll M^2$.

    Note that    
    $$ u_1...u_{\nu} \mid [u_1,...,u_{\nu}]\cdot \prod \limits_{i < j} (u_i,u_{j}),$$
    and
    $$ [u_1,...,u_{\nu}]\mid q.$$
    We get
     \begin{equation}
    \label{lemma5.3_proof_4}
        u_1...u_{\nu}\ll qM^{\frac{\nu (\nu-1)}{2}}.
    \end{equation}

    We now estimate the number of pairs $a_1,a_2$ in the sum \eqref{lemma5.3_conclusion_1} for which $u_1,...,u_\nu$ have particular values. We have
    \begin{equation}
    \label{lemma5.3_7}
        c_{i_j}a_1+d_{i_j}a_2 =u_jx_j\quad(j=1,...,\nu)
    \end{equation} 
    and any two of these equations in $a_1,a_2$ have their left hand sides linearly independent. Since $\nu \geq 3,$ we can regard $x_3,...x_\nu$ as functions of $x_1,x_2$, and we note that $x_1,x_2$ determine $a_1,a_2$ uniquely. Plainly
    \begin{equation}
        \label{upper_bound_x_1_x_2}
        |x_1| \ll \frac{Mq}{u_1}, \quad
    |x_2| \ll \frac{Mq}{u_2},
     \end{equation}
    since $1\leq a_1 \leq q, 1\leq a_2 \leq q.$
    
    The first two of the equations \eqref{lemma5.3_7}, together with the $j$-th equation $(j \geq 3)$, imply a linear relation of the form
    $$ c_1^{(j)}u_1x_1+c_2^{(j)}u_2x_2+c^{(j)}u_jx_j=0,$$
    where
    $$ c_1^{(j)}=c_{i_2}d_{i_j}-c_{i_j}d_{i_2},\  c_2^{(j)}=c_{i_j}d_{i_1}-c_{i_1}d_{i_j},\  c^{(j)}=c_{i_1}d_{i_2}-c_{i_2}d_{i_1},$$
    so that none of $c_1^{(j)}, c_2^{(j)}, c^{(j)}$ is 0. This gives a congruence to the modulo $u_j$ which must be satisfied by $x_1, x_2$, namely
    $$c_{1}^{(j)}u_1x_1+c_{2}^{(j)}u_2x_2 \equiv 0\ (\text{mod} \ u_j)\quad (j=3,...,\nu). $$
    
    The modulo $u_3,...,u_\nu$ of these congruences have only bounded common factors ($\ll M^2$) when taken in pairs, and have only bounded common factors ($\ll M^2$) with $u_1,u_2.$ Hence, for given $x_1,$ 
    $$c_{2}^{(j)}u_2x_2 \equiv -c_{1}^{(j)}u_1x_1\ (\text{mod} \ u_j)$$
    has at most $ \left(c_{2}^{(j)}u_2,u_j\right) \leq \left(c_{2}^{(j)},u_j\right)(u_2,u_j) \ll M^3$ solutions $\text{mod} \ u_j$ for $x_2.$ 
    Furthermore, by CRT, $x_2$ mod $u_3...u_{\nu}$ has at most $M^{3(\nu-2)}$ solutions. 
    By \eqref{upper_bound_x_1_x_2}, we deduce that the number of possibilities for $x_2$, for given $x_1$, is 
    $$
    \ll \frac{Mq}{u_2} \cdot \frac{M^{3(\nu-2)}}{u_3...u_{\nu}}.
    $$
    The number of possibilities for $x_1,x_2$ and so for $a_1,a_2$, is 
    $$\ll \frac{M^2q^2}{u_1u_2}
    \cdot \frac{M^{3(\nu-2)}}{u_3...u_{\nu}}.$$

     By \eqref{lemma5.3_proof_1} and \eqref{lemma5.3_proof_3}, the sum in \eqref{lemma5.3_conclusion_1} is
     \begin{equation*}
        \begin{split}
            &\ll \sum_{\substack{u_1,...,u_{\nu}\\[3pt]\eqref{lemma5.3_proof_4}}} M^{\frac{24}{5}}(\frac{q}{u_1})^{-\theta_1}...(\frac{q}{u_\nu})^{-\theta_\nu}\frac{M^2q^2}{u_1...u_\nu}M^{3(\nu-2)}\\
            &\ll M^{3\nu+4/5}q^{2-\theta_1-...-\theta_\nu}\sum_{\substack{u_1,...,u_{\nu}\\[3pt]\eqref{lemma5.3_proof_4}}} u_1^{\theta_1-1}...u_\nu^{\theta_\nu-1}.
        \end{split}
     \end{equation*}
    By \eqref{lemma5.3_proof_2} and the fact $\theta_j \leq 2$, this is 
    \begin{equation*}
        \begin{split}
            &\ll M^{3\nu+\frac{4}{5}}q^{-\frac{14}{5}}\sum_{\substack{u_1,...,u_{\nu}\\[3pt]\eqref{lemma5.3_proof_4}}}  u_1...u_{\nu}\\
            &\ll M^{3\nu+\frac{4}{5}}q^{-\frac{14}{5}}
            \left(qM^{\frac{\nu (\nu-1)}{2}}\right)^{1+\epsilon}\\
            &\ll q^{-\frac{9}{5}+\epsilon}M^{\frac{\nu^2}{2}+\frac{5\nu}{2}+\frac{4}{5}+\epsilon}.
        \end{split}
     \end{equation*}
   
    Considering the plain fact $\nu \leq 16,$ we finish the proof of \eqref{lemma5.3_conclusion_1}. 
    
    The proof of \eqref{lemma5.3_conclusion_2} is similar. The only difference is that \eqref{lemma5.3_proof_2} is replaced by
    $$\theta_1+...+\theta_\nu=\frac{16}{3} .$$ After the similar argument, we have
    $$
    \sum_{a_1,a_2} \frac{1}{(q_1...q_{16})^{\frac{1}{3}}}
    \ll q^{-\frac{7}{3}+\epsilon}M^{\frac{\nu ^2}{2}+\frac{5\nu}{2}+\frac{4}{3}+\epsilon}.
    $$
    Then \eqref{lemma5.3_conclusion_2} follows taking $\nu =16$ again.
    
\end{proof}

\begin{lemma}
    Let $\tau \geq 0,$ $\mu_i$ be given by \eqref{mu_i_definition}, and let 
    \begin{equation}
    \label{lemma5.4_conclusion_1}
        \mathscr{D}(\tau)=\left\{ (\beta_1,\beta_2):\max(|\beta_1|,|\beta_2|) > P^{\tau-3}\right\}.
    \end{equation}
    Then
    \begin{equation}
    \label{lemma5.4_conclusion_2}
        \iint \limits_{\mathscr{D}(\tau)} \prod _{i=1}^{12}\left(\frac{P}{1+P^3|\mu_i|}\right) d\beta_1d\beta_2 \ll P^{6-5\tau}M^{11}.
    \end{equation}
    \begin{equation}
    \label{lemma5.4_conclusion_3}
        \int_{-\infty}^{\infty} \int_{-\infty}^{\infty} \prod _{i=1}^{12}\left(\frac{P}{1+P^3|\mu_i|}\right) d\beta_1d\beta_2 \ll P^{6}M^{11}.
    \end{equation}
    and
    \begin{equation}
    \label{lemma5.4_conclusion_4}
        \int_{-\infty}^{\infty} \int_{-\infty}^{\infty} \prod{}^{'} \left(\frac{P}{1+P^3|\mu_i|}\right) d\beta_1d\beta_2 \ll P^{5}M^{10},
    \end{equation}
    where $\prod{}^{'}$ denotes a product over any eleven of $i=1,...,12.$
\end{lemma}
    
\begin{proof}
    It suffices to consider the part of $\mathscr{D}(\tau)$ for which $|\beta_1| \geq |\beta_2|,$ and then when $\beta_1$ is given, the range of integration for $\beta_2$  is $\ll |\beta_1|$.
    
    If $ \frac{c_i}{d_i} \neq \frac{c_j}{d_j}$, then \eqref{mu_i_definition} implies, on solving two of them, that
    \begin{equation}
    \label{lemma5.4_proof_1}
        \beta_1=\frac{\left | \begin{matrix}
        \mu_{i}& d_{i} \\
        \mu_{j}& d_{j} \\
        \end{matrix} \right |}{\left | \begin{matrix}
        c_{i}& d_{i} \\
        c_{j}& d_{j} \\
        \end{matrix} \right |} \ll M(|\mu_{i}|+|\mu_{j}|).
    \end{equation}
    We know that no more than five of the ratios
    \begin{equation}
    \label{lemma5.4_proof_2}
        \frac{c_1}{d_1},...,\frac{c_{12}}{d_{12}}
    \end{equation}
    can be equal. We divide the ratios \eqref{lemma5.4_proof_2} into blocks of equal ones, of length $l_1,...,l_{\nu},$ where
    $$ 5 \geq l_1 \geq l_2 \geq ... \geq l_\nu \geq 1,\quad l_1+...+l_\nu=12.$$
    If $i,j$ are indices from different blocks, then \eqref{lemma5.4_proof_1} tells us that 
    either $|\mu_i| \gg \frac{|\beta_1|}{M}$ 
    or $|\mu_j| \gg \frac{|\beta_1|}{M}$. Hence $|\mu_i| \gg \frac{|\beta_1|}{M}$ for all $i$ except possibly those in one particular block, and therefore
    \begin{equation*}
    \begin{split}
        \prod _{i=1}^{12}\left(\frac{P}{1+P^3|\mu_i|}\right)
        &\ll \prod _{i=1}^{12} \min \{P,P^{-2}|\mu_i|^{-1}\}\\
        &\ll P^{l_1}(P^{-2}M|\beta_1|^{-1})^{l_2+...+l_\nu}.\\
    \end{split}
    \end{equation*}
    It follows that the integral in \eqref{lemma5.4_conclusion_2} is
    \begin{equation*}
    \begin{split}
        &\ll \int_{P^{\tau-3}}^{\infty} P^{l_1}(P^{-2}M|\beta_1|^{-1})^{l_2+...+l_\nu} \beta_1 d\beta_1\\
        &\ll P^{l_1-2(l_2+...+l_\nu)}M^{l_2+...+l_\nu}P^{(3-\tau)(l_2+...+l_\nu-2)}\\
        &\ll P^{6+\tau(l_1-10)}M^{12-l_1}.\\
        &\ll P^{6-5\tau}M^{11}.
    \end{split}
    \end{equation*}
    
    If $\tau=0$ and the integral is extended over the whole plane, the estimate \eqref{lemma5.4_conclusion_2} remains valid. Since 
    \begin{equation*}
            \int_{-\infty}^{\infty} \int_{-\infty}^{\infty} \prod _{i=1}^{12}\left(\frac{P}{1+P^3|\mu_i|}\right) d\beta_1d\beta_2 =\left(\iint \limits_{\mathscr{D}(0)}+\iint \limits_{{\mathscr{D}(0)}^c}\right) \prod _{i=1}^{12}\left(\frac{P}{1+P^3|\mu_i|}\right) d\beta_1d\beta_2,\\
    \end{equation*}
    and 
    \begin{equation*}
            \iint \limits_{{\mathscr{D}(0)}^c} \prod _{i=1}^{12}\left(\frac{P}{1+P^3|\mu_i|}\right) d\beta_1d\beta_2 \ll P^6.\\
    \end{equation*}
    
    The proof of \eqref{lemma5.4_conclusion_4} is similar to that of \eqref{lemma5.4_conclusion_3}, and the only change in the argument is that $l_1+...+l_\nu=11$ instead of 12.
    
\end{proof}

We introduce $P_0,$ a parameter $\ll P^{4(1-\delta)/5}$  whose specific value will be determined in \eqref{final_P_0}.

\begin{lemma}
\label{lemma5.5}
    The contribution of all the major arcs $\mathfrak{M}(a_1,a_2,q)$ with $q>P_0$ to the integral \eqref{N(P)_definition} is 
    $$\ll P^{\frac{46}{5}+4\delta+\epsilon}M^{186}P_0^{-\frac{4}{5}}.
    $$
\end{lemma}

\begin{proof}
    By Lemma \ref{lemma5.1} and Lemma \ref{lemma5.2}, this part of contribution is
    \begin{equation*}
        \begin{split}
            &\ll \sum_{q>P_0}\sum_{a_1,a_2}  \  \iint \limits_{\mathfrak{M}(a_1,a_2,q)} \prod _{i=1}^{10}\frac{P}{{q_i}^{1/3}(1+P^3|\mu_i|)}\prod _{i=11,12}\frac{M^{8/3}P}{{q_i}^{1/3}(1+P^3|\mu_i|)}\prod _{i=13}^{16}\frac{P^{4/5+\delta}}{{q_i}^{1/5}} d\beta_1d\beta_2\\
            &\ll P^{16/5+4\delta}M^{16/3}\sum_{q>P_0}\sum_{a_1,a_2} \frac{1}{(q_1...q_{12})^{1/3}(q_{13}...q_{16})^{1/5}}\int_{-\infty}^{\infty} \int_{-\infty}^{\infty} \prod _{i=1}^{12}\left(\frac{P}{1+P^3|\mu_i|}\right) d\beta_1d\beta_2,
        \end{split}
    \end{equation*}
    where the summation of $a_1,a_2$ is over $1\leq a_1,a_2\leq q$ with $(a_1,a_2,q)=1$.
    
    From \eqref{lemma5.4_conclusion_3} and \eqref{lemma5.3_conclusion_1},
    the above is
    \begin{equation*}
        \begin{split}
            &\ll P^{\frac{16}{5}+4\delta}M^{\frac{16}{3}}
            \cdot \sum_{q>P_0} 
            q^{-\frac{9}{5}+\epsilon}M^{168+\frac{4}{5}}
            \cdot P^6M^{11}\\
            &\ll P^{\frac{46}{5}+4\delta+\epsilon}M^{186}P_0^{-\frac{4}{5}}.\\
        \end{split}
    \end{equation*}
    
\end{proof}

Let
\begin{equation}
\label{S_0_definition}
    \mathfrak{M}_0(a_1,a_2,q)=\left\{ (\alpha_1,\alpha_2): \left|\alpha_r-\frac{a_r}{q}\right|< P^{\delta-3}\quad (r=1,2)\right\}
\end{equation}
denote a contracted major arc.

\begin{lemma}
\label{lemma5.6}
    The contribution of all the $\mathfrak{M}(a_1,a_2,q)\setminus \mathfrak{M}_0(a_1,a_2,q)$ with $1 \leq a_1,a_2 \leq q \leq P_0$ and $(a_1,a_2,q)=1$ to the integral \eqref{N(P)_definition} is 
    $
    \ll P^{\frac{46}{5}-5\delta}M^{186}
    $.
\end{lemma}

\begin{proof}
    In this case, $q_i \leq P_0\leq P^{4(1-\delta)/5}$. Lemma 10 of Davenport \cite{Davenport1939} provides a better upper bound for \(U(\gamma_i)\) than invoking Weyl's inequality in Lemma \ref{lemma5.2}, specifically \(U(\gamma_i) \ll q_i^{-1/3}P^{4/5}\).

    Thus, by \eqref{lemma5.1_conclusion} and \eqref{lemma5.4_conclusion_1} for a particular set $a_1,a_2,q$, the contribution is
    \begin{equation*}
        \ll P^{\frac{16}{5}}M^{\frac{16}{3}}(q_1...q_{16})^{-\frac{1}{3}}\iint \limits_{\mathscr{D}(\delta)} \prod _{i=1}^{12}\left(\frac{P}{1+P^3|\mu_i|}\right) d\beta_1d\beta_2.
    \end{equation*}
    By \eqref{lemma5.3_conclusion_2} and \eqref{lemma5.4_conclusion_2},
    the contribution from the part of this lemma is
    \begin{equation*}
        \begin{split}
            &\ll P^{\frac{46}{5}-5\delta}M^{\frac{49}{3}} \sum_{q \leq P_0}\sum_{a_1,a_2} \frac{1}{(q_1...q_{16})^{1/3}}\\
            &\ll P^{\frac{46}{5}-5\delta+\epsilon}M^{185+\frac{2}{3}} \sum_{q \leq P_0} q^{-7/3}\\
            &\ll P^{\frac{46}{5}-5\delta}M^{186}.
        \end{split}
    \end{equation*}
    
\end{proof}

 \begin{lemma}
     The contribution of all the $\mathfrak{M}_0(a_1,a_2,q)$ with $1 \leq a_1,a_2 \leq q \leq P_0$ and $(a_1,a_2,q)=1$ to the integral \eqref{N(P)_definition} is
      \begin{equation}
      \label{lemma5.7_conclusion_1}\begin{split}
          \sum_{q\leq P_0}\sum_{a_1,a_2}\  \iint \limits_{\mathfrak{M}_{0}(a_1,a_2,q)} \prod_{i=1}^{12}\frac{S(b_i,q_i)I_i(\mu_i)}{q_i}&\prod_{i=13}^{16}U(\gamma_i)d\beta_1d\beta_2+O\left(P^{\frac{41}{5}+\epsilon}M^{\frac{46}{3}}P_0^{\frac{11}{3}}\right),
      \end{split}
    \end{equation}
    where
    \begin{equation}
       \label{I_i_definition}
       I_i(\gamma)=\int_{\xi_iP}^{\zeta_iP} e(\gamma x^3) dx,
    \end{equation}
    the summation of $a_1,a_2$ is the same as in Lemma \ref{lemma5.3}.
 \end{lemma}

 \begin{proof}
     We have obtained
     \begin{equation*}
        T_i(\gamma_i)=\frac{S(b_i,q_i)}{q_i}I_i(\mu_i)
        +O\left(q_i^{\frac{2}{3}+\epsilon}\right)
    \end{equation*}
    in \eqref{lemma5.1_5}.
    And \eqref{lemma5.1_conclusion} gives an upper bound of $T_i(\gamma_i)$, which is  greater than 
    $q_i^{\frac{2}{3}+\epsilon}$. 
    Hence
    \begin{equation}
    \label{lemma5.7_proof_1}
        \prod_{i=1}^{12}T_i(\gamma_i)-\prod_{i=1}^{12}\frac{S(b_i,q_i)I_i(\mu_i)}{q_i} \ll q^{\frac{2}{3}+\epsilon}M^{\frac{16}{3}}\prod{}^{'} \left(\frac{P}{1+P^3|\mu_i|}\right),
    \end{equation}
     where the product $\prod{}^{'}$ is the same as in \eqref{lemma5.4_conclusion_4}. By \eqref{lemma5.4_conclusion_4}, we have
     \begin{equation*}
     \begin{split}
         &\int_{-\infty}^{\infty} \int_{-\infty}^{\infty}
         \left|\prod_{i=1}^{12}T_i(\gamma_i)
         -\prod_{i=1}^{12}\frac{S(b_i,q_i)I_i(\mu_i)}{q_i}\right|
         d\beta_1d\beta_2\\
         \ll& q^{\frac{2}{3}+\epsilon}M^{\frac{16}{3}} \int_{-\infty}^{\infty} \int_{-\infty}^{\infty}\prod{}^{'}\left(\frac{P}{1+P^3|\mu_i|}\right) d\beta_1d\beta_2\\
         \ll& P^5M^{\frac{46}{3}}q^{\frac{2}{3}+\epsilon}.
     \end{split}
     \end{equation*}
    Noting the trivial estimate
    \begin{equation*}
        U(\gamma_i)\ll P^{\frac{4}{5}},
    \end{equation*}
    the contribution from the part in this lemma is
    \begin{equation*}
        \begin{split}
            &\quad \sum_{q\leq P_0}\sum_{a_1,a_2}\ \iint \limits_{\mathfrak{M}_{0}(a_1,a_2,q)} \prod_{i=1}^{12}T_i(\gamma_i)\prod_{i=13}^{16}U(\gamma_i)d\beta_1d\beta_2\\
            &=
            \sum_{q\leq P_0}\sum_{a_1,a_2}\ \iint \limits_{\mathfrak{M}_{0}(a_1,a_2,q)}\prod_{i=1}^{12}\frac{S(b_i,q_i) I_i(\mu_i)}{q_i}\cdot\prod_{i=13}^{16}U(\gamma_i)d\beta_1d\beta_2\\
            &+O\left(\sum_{q\leq P_0}\sum_{a_1,a_2}P^{\frac{16}{5}}\int_{-\infty}^{\infty} \int_{-\infty}^{\infty}\left|\prod_{i=1}^{12}T_i(\gamma_i)-\prod_{i=1}^{12}\frac{S(b_i,q_i) I_i(\mu_i)}{q_i}\right|d\beta_1d\beta_2\right),\\
        \end{split}
    \end{equation*}
    where the error term is
    \begin{equation*}
    \begin{split}
        &\ll \sum_{q\leq P_0}\sum_{a_1,a_2}P^{\frac{16}{5}}P^5M^{\frac{46}{3}}q^{\frac{2}{3}+\epsilon}
         \ll P^{\frac{41}{5}+\epsilon}M^{\frac{46}{3}}P_0^{\frac{11}{3}}.
    \end{split}
    \end{equation*}
 \end{proof}

\begin{lemma}
\label{lemma5.8}
    The contribution of all the $\mathfrak{M}_0(a_1,a_2,q)$ with $1 \leq a_1,a_2 \leq q \leq P_0$ and $(a_1,a_2,q)=1$ to the integral \eqref{N(P)_definition} is 
    \begin{equation}
    \label{lemma5.8_conclusion_1}
        P^{\frac{16}{5}}\mathfrak{S}(P_0)I(P)
+O\left(P^{\frac{43}{5}+3\delta}M^{\frac{19}{3}}P_0^{3}
    +
    P^{\frac{41}{5}+\epsilon}M^{\frac{46}{3}}P_0^{\frac{11}{3}}\right),
    \end{equation}
    where 
    \begin{equation}
    \label{truncated_S_definition}
       \mathfrak{S}(P_0)=\sum_{q\leq P_0}\mathop{\sum_{a_1=1}^{q}\sum_{a_2=1}^{q}}_{(a_1,a_2,q)=1} \prod_{i=1}^{16} \frac{S(b_i,q_i)}{q_i},
    \end{equation}
    \begin{equation}
    \label{I(P)_definition}
       I(P)=\iint \limits_{|\beta_1|,|\beta_2|<P^{\delta-3}} I_1(\mu_1)...I_{12}(\mu_{12}) d\beta_1d\beta_2.
    \end{equation}
    
\end{lemma}

\begin{proof}
    By Lemma 8 of Davenport \cite{Davenport1939}, we have
    \begin{equation*}
        U(\gamma_i)=\frac{S(b_i,q_i)}{q_i}
        \sum_{\left(P^{\frac{4}{5}}\right)^3 \leq n \leq \left(2P^{\frac{4}{5}}\right)^3} \frac{1}{3}n^{-\frac{2}{3}}e(\mu_i n)
        +O\left(q_i^{\frac{2}{3}+\epsilon}\right).
    \end{equation*}
    Invoking Euler-Maclaurin summation formula,
    we have
    \begin{equation*}
    \begin{split}
        &\sum_{P^{12/5} \leq n \leq 2^3 P^{12/5}} \frac{1}{3}n^{-2/3}e(\mu_i n) \\
        &= \int_{P^{12/5}}^{8P^{12/5}} 
        \frac{1}{3}\eta^{-2/3}e(\mu_i\eta)\mathrm{d}\eta  
        + \frac{1}{6}\left( n^{-2/3}e(\mu_i n) \Big|_{P^{12/5}}^{8P^{12/5}} \right)  
        + \int_{P^{12/5}}^{8P^{12/5}} 
        B_1(t)\left( \frac{1}{3}t^{-2/3}e(\mu_i t) \right)\mathrm{d}t \\
        &= \int_{P^{12/5}}^{8P^{12/5}}
        \frac{1}{3}\eta^{-2/3}e(\mu_i\eta)\mathrm{d}\eta 
        + O\left(P^{-8/5}\right).
    \end{split}
\end{equation*}
    So we get
    \begin{equation*}
        U(\gamma_i)=\frac{S(b_i,q_i)}{q_i}J(\mu_i)+O\left(q^{\frac{2}{3}+\epsilon}\right),
    \end{equation*}
    where
    \begin{equation*}
    J(\mu_i) = \int_{P^{12/5}}^{8P^{12/5}} \frac{1}{3}\eta^{-2/3}e(\mu_i\eta)\mathrm{d}\eta.
    \end{equation*}
    Recalling the definition of $\mathfrak{M}_0(a_1,a_2,q)$ in \eqref{S_0_definition}, and $\mu_i =c_i\beta_1+d_i\beta_2$,
    we have
    \begin{equation*}
        \mu_i \ll MP^{\delta-3}, \quad \mu_i \eta \ll M P^{\delta-\frac{3}{5}}\ll 1,
    \end{equation*}
    for $P^{\frac{12}{5}} \leq \eta \leq 8P^{\frac{12}{5}}.$
    Hence, substituting
    $$
    e(\mu_i\eta)=1+O\left(|\mu_i|P^{\frac{12}{5}}\right),
    $$
    we get
    \begin{equation*}
    \begin{split}
        J(\mu_i)
        &=\frac{1}{3}
        \int_{P^{\frac{12}{5}}}^{8P^{\frac{12}{5}}} \eta^{-\frac{2}{3}}d\eta+O\left(|\mu_i|P^{\frac{16}{5}}\right)\\
        &=P^{\frac{4}{5}}+O\left(P^{\frac{1}{5}+\delta}M\right).
    \end{split}
    \end{equation*}
    It follows that
    \begin{equation*}
        \prod_{i=13}^{16}U(\gamma_i)=P^{\frac{16}{5}}
        \prod_{i=13}^{16}\frac{S(b_i,q_i)}{q_i}
        +O\left(P^{\frac{13}{5}+\delta}M \right).
    \end{equation*}
    Substituting this in \eqref{lemma5.7_conclusion_1}, we obtain the main term in \eqref{lemma5.8_conclusion_1}, together with an error term. The error term is \begin{equation*}
        \ll P^{\frac{13}{5}+\delta}M\sum_{q\leq P_0}\sum_{a_1,a_2}\ \iint \limits_{\mathfrak{M}_{0}(a_1,a_2,q)} \prod_{i=1}^{12}\frac{S(b_i,q_i)}{q_i}I_i(\mu_i)d\beta_1d\beta_2
        +
        P^{\frac{41}{5}+\epsilon}M^{\frac{46}{3}}P_0^{\frac{11}{3}}.
    \end{equation*}
    
    By \eqref{lemma5.1_conclusion} and \eqref{lemma5.1_5}, we have
    \begin{equation*}
        \prod_{i=1}^{12}\frac{S(b_i,q_i)}{q_i}I_i(\mu_i) \ll P^{12}M^{\frac{16}{3}}.
    \end{equation*}
    And using the trivial bound 
    \begin{equation*}
        |\mathfrak{M}_{0}(a_1,a_2,q)| \leq P^{2(\delta-3)},
    \end{equation*}
    the error term is 
    $
    \ll P^{\frac{43}{5}+3\delta}M^{\frac{19}{3}}P_0^{3}
    +
    P^{\frac{41}{5}+\epsilon}M^{\frac{46}{3}}P_0^{\frac{11}{3}}.
    $ 
    
\end{proof}

\section{The lower bound of the main term}

In this section, we focus on estimating the expected main term $P^{\frac{16}{5}}\mathfrak{S}(P_0)I(P)$. The bulk is to get a lower bound for the truncated singular series $\mathfrak{S}(P_0).$

\begin{lemma}
\label{lem:measure_lower_bound}
Let
\[
\mathcal{B}=\prod_{i=1}^n [a_i,b_i]\subset \mathbb R^n
\]
be an axis-parallel box, where $
\ell_i:=b_i-a_i > 0,$ and let
$\mathcal S \subset \mathbb{R}^n$ be an affine subspace of dimension $1 \le k \le n$ passing through the centre of $\mathcal B$. 
We denote by
\[
\operatorname{vol}_{k}(\mathcal S\cap\mathcal B)
\]
the $k$-dimensional Lebesgue measure on $S$.
Then we have
\[
\operatorname{vol}_{k}(S\cap\mathcal B)\ \ge\ \,\prod_{t=1}^{k}\ell_{(t)},
\]
where $\ell_{(1)}\le \ell_{(2)}\le\cdots\le \ell_{(n)}$ is the non-decreasing rearrangement of
$\{\ell_1,\dots,\ell_n\}$.
\end{lemma}

\begin{proof}
Let $x_0$ be the centre of $\mathcal B$. WLOG, we may assume $x_0=0$ and $\mathcal S$ is a linear
subspace of dimension $k$.

Set the unit-volume cube $\mathcal Q=[-1/2,1/2]^n$ and let 
\begin{equation}
\label{eq: D def}
    D=\mathrm{diag}(\ell_1,\dots,\ell_n).
\end{equation}
Then $\mathcal B = D \mathcal Q$. Define the $k$-dimensional subspace $\mathcal F=D^{-1}\mathcal S$.
The map $D$ restricts to a linear isomorphism $\mathcal F\to \mathcal S$, and
\[
\mathcal S\cap D \mathcal Q = D(\mathcal F\cap \mathcal Q).
\]
Let $J_{k}(D|_{\mathcal F})$ denote the $k$-dimensional Jacobian of
the restriction $D|_{\mathcal F}: \mathcal F\to \mathcal S$,
\begin{equation}
\label{eq: Jk_def}
J_{k}(D|_{\mathcal F})=\sqrt{\det\!\big((D U)^{\mathsf T}(D U)\big)}, 
\end{equation}
where $U\in\mathbb{R}^{n\times k}$ has orthonormal columns spanning $\mathcal F$.
We have
\begin{equation}\label{eq:cov-subspace}
\operatorname{vol}_{k}(\mathcal S\cap D \mathcal Q)
= J_{k}(D|_{\mathcal F})\ \operatorname{vol}_{k}(\mathcal F\cap \mathcal Q).
\end{equation}
Now we invoke Vaaler's Corollary \cite[page 544]{Vaaler1979} to obtain
\[
\operatorname{vol}_k(\mathcal F\cap \mathcal Q)\ge 1.
\]
Combining this with \eqref{eq: Jk_def} and \eqref{eq:cov-subspace} gives
\[
\operatorname{vol}_{k}(\mathcal S\cap D \mathcal Q)\ \ge\ \sqrt{\det\!\big((D U)^{\mathsf T}(D U)\big)}.
\]

It remains to show
\begin{equation*}
 \det\!\big((D U)^{\mathsf T}(D U)\big) \ge    \prod_{t=1}^{k}\ell_{(t)}^2,
\end{equation*}
uniformly in $U$, where $U\in\mathbb{R}^{n\times k}$ has orthonormal columns and $D$ is given by the \eqref{eq: D def}.
This can be proved by Courant–Fischer–Weyl min-max principle and we omit the process.
This completes the proof. 

\end{proof}

\begin{lemma}
\label{I(P)_lower_bound}
  When
  \begin{equation}
  \label{lemma6.1_1}
      P>M^{\frac{31}{3\delta}},
  \end{equation}
  we have
  \begin{equation}
  \label{I(P)}
      I(P) \gg P^6 M^{-\frac{26}{3}},
  \end{equation}
  where $I(P)$ is defined by \eqref{I(P)_definition}.
\end{lemma}

\begin{proof}
 Recall the definition \eqref{I_i_definition} of $I_i(\gamma),$ by \eqref{lemma5.1_6}, we have
 \begin{equation*}
        I_i(\mu_i) \ll \frac{PM^{\frac{8}{3}}}{(1+P^3|\mu_i|)}.
    \end{equation*}
   Let 
   \begin{equation*}
       I_0(P)=\int_{-\infty}^{\infty} \int_{-\infty}^{\infty}I_1(\mu_1)...I_{12}(\mu_{12})d\beta_1d\beta_2,
   \end{equation*}
  then  
  \begin{equation*}
       |I(P)-I_0(P)| \ll M^{32}\iint \prod_{i=1}^{12} \frac{P}{1+P^3|\mu_i|},
   \end{equation*}
    extended over $\max\{|\beta_1|,|\beta_2|\}>P^{\delta-3}$. Recalling \eqref{lemma5.4_conclusion_2}, we have
   \begin{equation*}
       |I(P)-I_0(P)| \ll P^{6-5\delta}M^{43}.
   \end{equation*}

   After the substitution $t=\left(\frac{x}{P}\right)^3$, we obtain
   \begin{equation*}
       \begin{split}
           I_i(\mu_i)
           &=\int_{\xi_iP}^{\zeta_iP} e(\mu_i x^3) dx =\frac{P}{3} \int_{\frac{1}{8}\eta_i}^{8\eta_i} \frac{e(tP^3\mu_i)}{t^{\frac{2}{3}}}dt.
       \end{split}
   \end{equation*}
    We put $\omega_1=P^3\beta_1$ and $\omega_2=P^3\beta_2$.
    Then
    \begin{equation*}
        I_0(P)=\frac{P^6}{3^{12}}\int_{-\infty}^{\infty} \int_{-\infty}^{\infty} \left( \int_{\mathfrak{B}} \frac{e(L_1\omega_1+L_2\omega_2)}{(t_1...t_{12})^{\frac{2}{3}}} d\mathbf{t} \right) d\omega_1d\omega_2,
    \end{equation*}
    where $\mathfrak{B} \subset \mathbb{R}^{12}$ is the box defined by
    \begin{equation*}
        \frac{1}{8}\eta_i <t_i<8\eta_i \quad (i=1,...,12),
    \end{equation*}
    and where
    \begin{equation*}
    \begin{split}
        L_1=L_1(\mathbf{t})=c_1t_1+...+c_{12}t_{12},\\
        L_2=L_2(\mathbf{t})=d_1t_1+...+d_{12}t_{12}.
    \end{split}
    \end{equation*}
    The equations $L_1(\mathbf{t})=L_2(\mathbf{t})=0$ define a 10-dimensional linear space $\mathcal{S} \subset \mathbb{R}^{12}$, which passes through the point $\boldsymbol{\eta}$ of Lemma \ref{lemma3.1}.

    Define
    \begin{equation*}
        K_\lambda(u)=\int_{-\lambda}^{\lambda} e(u\omega)\,d\omega.
    \end{equation*}
    By Fourier's integral formula, $K_\lambda$ converges in the sense of distributions
to the Dirac delta distribution $\delta$, namely for any Schwartz function $V$,
    \begin{equation*}
    \lim_{\lambda\to\infty}\int_{\mathbb R} V(u)\,K_\lambda(u)\,du = V(0).
    \end{equation*}
Let
\[
I_\lambda=\int_{-\lambda}^{\lambda}\int_{-\lambda}^{\lambda}
\left(\int_{\mathfrak B}\frac{e\!\left(L_1(\mathbf t)\omega_1+L_2(\mathbf t)\omega_2\right)}
{(t_1\cdots t_{12})^{2/3}}\,d\mathbf t\right)\,d\omega_1\,d\omega_2.
\]
Then 
\begin{equation*}
        I_0(P)=\frac{P^6}{3^{12}}\lim_{\lambda\to\infty} I_\lambda.
    \end{equation*}
By Fubini,
\[
I_\lambda=\int_{\mathfrak B}\frac{K_\lambda(L_1(\mathbf t))\,K_\lambda(L_2(\mathbf t))}
{(t_1\cdots t_{12})^{2/3}}\,d\mathbf t.
\]
Let
\[
A =
\begin{pmatrix}
\nabla L_1 \\
\nabla L_2
\end{pmatrix}
=
\begin{pmatrix}
c_1,...,c_{12} \\
d_1,...,d_{12}
\end{pmatrix},
\]
which is a $2\times 12$ matrix. Define
\[
J = \|\nabla L_1 \wedge \nabla L_2\| = \sqrt{\det(AA^{T})},
\]
noting that $J \ll M^2.$
Letting $\lambda\to\infty$ and using $K_\lambda\to\delta$ in $\mathcal D'(\mathbb R)$ twice, we get 
 \begin{equation*}
    \begin{split}        \lim_{\lambda\to\infty} I_\lambda
       =&
        \int_{\mathfrak B}\frac{\delta(L_1(\mathbf t))\,\delta(L_2(\mathbf t))}
{(t_1\cdots t_{12})^{2/3}}\,d\mathbf t\\
 =&\frac{1}{J}
        \int_{\mathfrak B \cap \mathcal{S}}\frac{1}
{(t_1\cdots t_{12})^{2/3}}\,d\operatorname{vol}_{10}(\mathbf t).
    \end{split}
    \end{equation*}

    The bound of $\eta_i$ in Lemma \ref{lemma3.1} gives
    \begin{equation*}
        \frac{1}{(t_{11}t_{12})^{\frac{2}{3}}} \gg M^{-\frac{8}{3}}.
    \end{equation*}
    On the other hand, by Lemma \ref{lem:measure_lower_bound}, $\operatorname{vol}_{10}(\mathfrak B \cap \mathcal S) \ge M^{-4}.$
    So
     \begin{equation*}
        \lim_{\lambda\to\infty} I_\lambda \gg M^{-\frac{26}{3}},
    \end{equation*}
    hence
    \begin{equation*}
        I_0(P) \gg P^6 M^{-\frac{26}{3}}.
    \end{equation*}
    
    The condition \eqref{lemma6.1_1} ensures the error term $|I(P)-I_0(P)|$ is relatively small and we finish the proof.
    
\end{proof}

Next, we estimate the lower bound of $\mathfrak{S}(P_0)$. 

We define
\begin{equation}
\label{A(q)}
    A(q)=\mathop{\sum_{a_1=1}^{q}\sum_{a_2=1}^{q}}_{(a_1,a_2,q)=1} \prod_{i=1}^{16}\frac{S(b_i,q_i)}{q_i},
\end{equation}
where $b_i, q_i$ are given by \eqref{b_i_definition}, \eqref{q_i_definition}. It's easy to verify $A(q)$ is multiplicative.

Recalling the definition \eqref{truncated_S_definition} of $\mathfrak{S}(P_0),$
we have
\begin{equation*}
    \mathfrak{S}(P_0)=\sum_{q\leq P_0}A(q).
\end{equation*}

We define
\begin{equation}
\label{lemma6.1_3}
    \rho(p^k)= \# \{ \mathbf{x} \in (\mathbb{Z}/p^k\mathbb{Z})^{16} : F(\mathbf{x})\equiv 0\text{  and }G(\mathbf{x})\equiv0, \pmod{p^k}\},
\end{equation}
\begin{equation}
\label{lemma6.1_4}
    \rho_{nz}(p^k)= \# \{ \mathbf{x} \in (\mathbb{Z}/p^k\mathbb{Z})^{16} : \prod_{i=1}^{n}x_i \neq 0,\  F(\mathbf{x})\equiv 0\text{  and }G(\mathbf{x})\equiv 0,\pmod{p^k}\},
\end{equation}
\begin{equation}
\label{lemma6.1_5}
\begin{split}
    \rho_i(p^k)= \# \{ \mathbf{x} \in (\mathbb{Z}/p^k\mathbb{Z})^{16} : x_i=0, F(\mathbf{x})\equiv 0\text{  and }G(\mathbf{x})\equiv 0,\pmod{p^k}\}.
\end{split}
\end{equation}

Using standard methods, we know
\begin{equation}
    \sum_{i=0}^{k}A(p^i)=p^{-14k}\rho(p^k).
\end{equation}

We introduce the truncated Euler product
\begin{equation}
\label{S(P_0)}
    S(P_0)=\prod_{p \leq P_0}\sum_{i=0}^{k(p)}A(p^i)=\prod_{p \leq P_0}p^{-14k(p)}\rho(p^{k(p)}),
\end{equation}
where
\begin{equation*}
    k(p)=\text{max}\{ t \in \mathbb{N} : p^t \leq P_0\}.
\end{equation*}

Next, we aim to determine a lower bound for this quantity. Fundamentally, our task lies in estimating $\rho(p^k)$. Our strategy is to handle the case of mod $p$ first, then solve the case of mod $p^k$ by lifting.

We introduce a lemma of Hooley \cite[Theorem 2]{Hooley1991}.
\begin{lemma}[Hooley, 1991]
\label{lemma_from_Hooley}
    Let $\mathcal{V}$ be a projective complete intersection over $\mathbb{F}_q$ in an ambient space of $n_1$ dimensions and let it have dimension $n$ and a singular locus of dimension $d$. Then the number of points on $\mathcal{V}$ having components in $\mathbb{F}_q$ equals
    \[
    \frac{q^{n + 1}-1}{q - 1}+O(q^{(n + d+1)/2}),
    \]
    where the constant implied by the $O$-notation depends at most on $n_1$ and the multi degree of $\mathcal{V}$.
\end{lemma}

When the coefficient matrix satisfies certain conditions, we can establish an acceptable lower bound of $\rho(p)$.

\begin{lemma}[The estimation of \(\rho(p)\)]
\label{lemma6.2}
  If for any prime \( p \leq M^2 \) satisfying \( p \equiv 1 \pmod{3} \), the system of equations  \eqref{eq:zuikaishi} admits at most 11 pairwise parallel vectors \( \begin{pmatrix} c_i \\ d_i \end{pmatrix} \in \mathbb{F}_p^2 \), then the singular locus has dimension at most $11$. Furthermore, we have $\rho(p) >0$, and
  \begin{equation}
  \label{lemma6.2_modp_conclusion}
      \rho(p) \geq p^{14}+O(p^{13}).
  \end{equation}
\end{lemma}

\begin{proof}
    Note that $\rho(p) >0$ is a direct result of the Chevalley--Warning theorem \cite[Theorem 1.6]{TimBrowning2021}. Next we explain \eqref{lemma6.2_modp_conclusion}.
    
    If $ p \equiv 2 \pmod{3}$ or $\ p=3$, then every element in $\mathbb{Z}_p$ is a cube. So the simultaneous equations \eqref{eq:zuikaishi} are essentially linear and we have $\rho(p) = p^{14}$. There are more solutions when the rank is less than full.
    
    If $ p \equiv 1\pmod{3}$, according to the Jacobian criterion, we know the dimension of the algebraic set determined by \eqref{eq:zuikaishi} over the finite field \(\mathbb{F}_p\) is 14. Then under the condition that the dimension \(d\) of the singular locus is \(d\leq11\), the expected estimate can be obtained by Lemma \(\ref{lemma_from_Hooley}\).  
    The singular locus is defined as the set of points $\mathbf{x} \in \mathbb{F}_p^{16}$ satisfying both equations \eqref{eq:zuikaishi}  mod $p$ are zeros at $\mathbf{x}$, and
    the Jacobian matrix has rank less than 2 at $\mathbf{x}$.
    The Jacobian matrix is
    $$
    J(\mathbf{x}) =
    \begin{bmatrix}
    c_1 x_1^2 & \cdots & c_{16} x_{16}^2 \\
    d_1 x_1^2 & \cdots & d_{16} x_{16}^2
    \end{bmatrix}.
    $$
    It suffices to discuss the rank of $J(\mathbf{x})$. The rank is less than 2 if and only if the two rows are linearly dependent, that is, there exists $(\lambda, \mu) \in \mathbb{F}_p^2,\   (\lambda, \mu)\not\equiv (0,0)$ such that:
    $$
    (\lambda c_i-\mu d_i)x_i^2 \equiv 0 \pmod{p},\quad \forall  i.
    $$
    From conditions we know there are at most 11 indices i satisfying $(\lambda c_i-\mu d_i) \equiv 0,\pmod{p}$. So $d\leq 11,$ and \eqref{lemma6.2_modp_conclusion} follows.
\end{proof}

\begin{remark}
\label{remarek_1}
    We can also employ cubic character to represent $\rho (p)$ for $p \equiv 1 \pmod{3}$. Specifically,
    \begin{equation*}
    \rho(p)=\sum_{\substack{c_1y_1+...c_ny_n \equiv 0\\[3pt]d_1y_1+...d_ny_n \equiv 0}} \prod_{i=1}^{n} \left(1+\chi(y_i)+\chi^2(y_i)\right),
        \end{equation*}
    where $\chi$ is a cubic character mod $p$ satisfying $\chi^ 3=id.$        Moreover, by employing the detecting summation and leveraging conclusions regarding the Gauss sum, we are also able to obtain \eqref{lemma6.2_modp_conclusion}. Nevertheless, under more stringent restrictive conditions, the number of parallel column vectors is at most 10.

    The part (i) of the $M$-good condition is limited to the range of $p \leq M^2$, because we assume that no ratio  $r_{i}$ is repeated more than six times in $\mathbb{Z}$ after Lemma \ref{lemma2.1}. When $p > M^2$, a non-zero determinant 
    $
    \left | \begin{matrix}
         c_{i}& d_{i} \\
          c_{j}& d_{j} \\
        \end{matrix} \right |
    $ in $\mathbb{Z}$ is not 0 mod $p,$ either.

For ordinary simultaneous equations \eqref{eq:zuikaishi} with \(p \equiv 1 \pmod{3}\), we cannot expect a conclusion like \eqref{lemma6.2_modp_conclusion} without imposing appropriate conditions on the coefficients \(c_i, d_i\). 
As a counterexample, consider the system
\[
\begin{cases}
x_1^3 \equiv 0,\\
x_2^3 - k x_3^3 \equiv 0,
\end{cases}
\pmod{p},
\]
where \(k\) is a non-cube modulo \(p\). In this case, it is necessary that \(x_1 = x_2 = x_3 = 0\), so
$$
\rho(p) \ll p^{n-3},
$$ 
and consequently,
\begin{equation}
\label{remark_rho(p)_conterexample}
    \frac{\rho(p^k)}{p^{14k}} \ll \frac{1}{p}.
\end{equation}
As will be clarified in the remark following Lemma~\ref{lemma_S_lower_bound}, \eqref{remark_rho(p)_conterexample} prevents us from obtaining a satisfactory lower bound for the main term. This highlights that our application of the circle method cannot cover all general cases.
\end{remark}

Before establishing the lemmas modulo \( p^k \), 
we briefly outline the idea behind the lifting process. 

Suppose we have already found a solution \(\mathbf{x} = (x_1, \dots, x_n)\) to the system of equations \eqref{eq:zuikaishi} modulo \(p^k\). Generally, we aim to find \(p\) solutions of the form \(\mathbf{x} + p^n \mathbf{z}\) to \eqref{eq:zuikaishi} modulo \(p^{k+1}\). This amounts to solving the following system of linear congruences in the vector \(\mathbf{z}\):
\[
3\begin{pmatrix}
c_1 x_1^2 & \cdots & c_n x_n^2 \\
d_1 x_1^2 & \cdots & d_n x_n^2
\end{pmatrix}
\mathbf{z}
\equiv -
\begin{pmatrix}
\frac{c_1 x_1^3 + \cdots + c_n x_n^3}{p^n} \\
\frac{d_1 x_1^3 + \cdots + d_n x_n^3}{p^n}
\end{pmatrix} \pmod{p}.
\]
We aim to find as many solutions as possible, ideally allowing \(n - 2\) coordinates of \(\mathbf{z}\) to range freely over \(\mathbb{Z}_p^{n-2}\), while the remaining two coordinates are solvable modulo \(p\).

It is clear the above argument is valid only when \(p \neq 3\), and additional adjustments are required in the case \(p = 3\).

\begin{lemma}[lifting for $p \neq 3$]
\label{lemma6.3}
  For any prime $p\neq 3$, and any integer $n \geq 1$, any solution $(x_i,x_j)$ of the simultaneous equations
  \begin{equation*}
    \begin{cases}
        c_ix_i^3+c_jx_j^3 \equiv \lambda,\\
        d_ix_i^3+d_jx_j^3 \equiv \mu,
    \end{cases}
    \pmod{p^n},
  \end{equation*}
  if 
  \begin{equation}
        \label{lemma6.3_conclusion_1}
          x_ix_j\left | \begin{matrix}
         c_{i}& d_{i} \\
          c_{j}& d_{j} \\
        \end{matrix} \right | \not\equiv 0\pmod{p},
    \end{equation}
  then there exist $(y_i,y_j)$, satisfying
  \begin{equation}
  \label{lemma6.3_conclusion_2}
    \begin{cases}
        c_iy_i^3+c_jy_j^3 \equiv \lambda,\\
        d_iy_i^3+d_jy_j^3 \equiv \mu,
    \end{cases}
    \pmod{p^{n+1}}
  \end{equation}
  and
  \begin{equation*}
    \begin{cases}
        y_i \equiv x_i,\\
        y_j \equiv x_j,
    \end{cases}
    \pmod{p}.
  \end{equation*}
\end{lemma}

\begin{proof}
    Let
    \begin{equation*}
    \begin{cases}
        y_i = x_i+K_ip^n,\\
        y_j = x_j+K_jp^n.
    \end{cases}
  \end{equation*}
  Substituting this in \eqref{lemma6.3_conclusion_2}, we obtain linear simultaneous equations about $K_i,K_i$,
  \begin{equation}
  \label{lemma6.3_proof_1}
    \left ( \begin{matrix}
       3c_{i}x_i^2& 3c_{j}x_j^2 \\
      3d_{i}x_i^2& 3d_{j}x_j^2 \\
        \end{matrix} \right) 
     \left ( \begin{matrix}
     K_i \\
     K_j \\
     \end{matrix} \right)
        =-
        \left ( \begin{matrix}
        \frac{c_{i}x_i^3+c_{j}x_j^3-\lambda}{p^{n}} \\
        \frac{d_{i}x_i^3+d_{j}x_j^3-\mu}{p^{n}} \\
        \end{matrix} \right),\quad(\text{mod }p).
  \end{equation}
  Noting that $(p,3)=1$ and \eqref{lemma6.3_conclusion_1}, we know \eqref{lemma6.3_proof_1} has an unique solution $K_i, K_j$ mod $p$.
  
\end{proof}

\begin{lemma}[lifting for $p =3$]
\label{lemma6.4}

  For $p= 3$, and any integer $n \geq 2$, any solution $(x_i,x_j)$ of the simultaneous equations
  \begin{equation*}
    \begin{cases}
        c_ix_i^3+c_jx_j^3 \equiv \lambda,\\
        d_ix_i^3+d_jx_j^3 \equiv \mu,
    \end{cases}
    \pmod{3^n},
  \end{equation*}
  if 
  \begin{equation}
        \label{lemma6.4_conclusion_1}
          x_ix_j\left | \begin{matrix}
         c_{i}& d_{i} \\
          c_{j}& d_{j} \\
        \end{matrix} \right | \not\equiv 0\pmod{3},
    \end{equation}
  then there exist $(y_i,y_j)$, satisfying
  \begin{equation}
  \label{lemma6.4_conclusion_2}
    \begin{cases}
        c_iy_i^3+c_jy_j^3 \equiv \lambda,\\
        d_iy_i^3+d_jy_j^3 \equiv \mu,
    \end{cases}
    \pmod{3^{n+1}}
  \end{equation}
  and
  \begin{equation*}
    \begin{cases}
        y_i \equiv x_i,\\
        y_j \equiv x_j.
    \end{cases}
    \pmod{3}
  \end{equation*}
\end{lemma}

\begin{proof}
    Unlike the treatment of Lemma \ref{lemma6.3}, here we let
    \begin{equation*}
    \begin{cases}
        y_i = x_i+K_i3^{n-1},\\
        y_j = x_j+K_j3^{n-1}.
    \end{cases}
  \end{equation*}
  Substituting this in \eqref{lemma6.4_conclusion_2}, we obtain linear simultaneous equations about $K_i,K_i$,
  \begin{equation}
  \label{lemma6.4_proof_1}
    \left ( \begin{matrix}
       c_{i}x_i^2& c_{j}x_j^2 \\
       d_{i}x_i^2& d_{j}x_j^2 \\
        \end{matrix} \right) 
     \left ( \begin{matrix}
     K_i \\
     K_j \\
     \end{matrix} \right)
        =-
        \left ( \begin{matrix}
        \frac{c_{i}x_i^3+c_{j}x_j^3-\lambda}{3^{n}} \\
        \frac{d_{i}x_i^3+d_{j}x_j^3-\mu}{3^{n}} \\
        \end{matrix} \right),\quad(\text{mod }3).
  \end{equation}
  Noting \eqref{lemma6.4_conclusion_1}, we know that \eqref{lemma6.4_proof_1} has a solution.
  
\end{proof}

\begin{remark}
    In Lemma \ref{lemma6.4}, we use a different way to get $y_i,y_j$. In fact, if we adopt the similar way as in Lemma \ref{lemma6.3} to lift, the determinant of \eqref{lemma6.4_proof_1} is 0 and the lifting interrupts. Due to this slight difference, for $p=3$, the induction starts with $n=2$, that is mod 9.
\end{remark}

Using lifting methods, it is relatively easy to handle when $p$ is large enough.

\begin{lemma}[lower bound of $\rho(p^k)$ when $p \geq p_0$]
\label{lower_bound_rho(p^k)_p_geq_p_0}
  If $n \geq 16,$ and the simultaneous equations  \eqref{eq:zuikaishi} satisfy $M$-good (i) (ii), then there exists a constant $p_0>3$ that is at most dependent on $n,$ and satisfies for any $k \geq 1$ and any prime $p \geq p_0 $, we have 
  \begin{equation}
  \label{lemma6.5_conclusion_1}
      \rho(p^k) \geq p^{14k}\left(1+O\left(\frac{1}{p}\right)\right).
  \end{equation}
\end{lemma}

\begin{proof}
    When $k=1$, noting the conditions of Lemma \ref{lemma6.2} are covered by $M$-good (i). So Lemma \ref{lemma6.2} ensures that \eqref{lemma6.5_conclusion_1} holds. However, for the induction process to proceed when \(k > 1\), we need to prune $\rho(p)$. Since it is difficult to find a solution \(\mathbf{x} = (x_1, \dots, x_n) \pmod{p}\) in which exactly two coordinates \(x_i, x_j\) are non-zero modulo \(p\) and the corresponding determinant  
    \[
    \left| \begin{matrix}
    c_i & c_j \\
    d_i & d_j \\
    \end{matrix} \right| \not\equiv 0 \pmod{p}
    \]  
    holds, we instead roughly consider lifting solutions in which all coordinates are non-zero modulo \(p\).
    Recalling the definitions of $\rho(p^k), \rho_{nz}(p^k), \rho_i(p^k)$  in \eqref{lemma6.1_3},  \eqref{lemma6.1_4}, \eqref{lemma6.1_5}, a simple application of the inclusion - exclusion principle yields the following inequality:
\begin{equation*}
\rho_{nz}(p) \geq \rho(p) - \sum_{i = 1}^{16} \rho_i(p).
\end{equation*}
    Moreover, considering the trivial upper bound $\rho_i(p) \leq 3^n p^{13} \ll p^{13}$, it follows that
    \begin{equation}
    \label{lemma6.5_proof_1}
        \rho_{nz}(p) \geq p^{14}\left(1+O\left(\frac{1}{p}\right)\right)>0,\quad \text{when }p\geq p_0,
    \end{equation}
    where $p_0$ is an absolute constant that is at most dependent on $n$.
    
    When $k>1$ and $p \geq p_0$, we consider the simultaneous equations of $x_i,x_j$,
    \begin{equation}
    \label{lemma6.5_proof_2}
    \begin{cases}
        c_{1}x_1^3+...+c_{16}x_{16}^3 \equiv 0,\\
        d_1x_1^3+...+d_{16}x_{16}^3 \equiv 0 ,\\
    \end{cases} \pmod{p}
  \end{equation}
  or after appropriate transfer,
    \begin{equation}
    \label{lemma6.5_proof_3}
    \begin{cases}
        c_ix_i^3+c_jx_j^3 \equiv -(c_{1}x_1^3+...+\hat{c_{i}x_i^3}+...+\hat{c_{j}x_j^3}+c_{16}x_{16}^3),\\
        d_ix_i^3+d_jx_j^3 \equiv -(d_{1}x_1^3+...+\hat{d_{i}x_i^3}+...+\hat{d_{j}x_j^3}+d_{16}x_{16}^3),\\
    \end{cases} \pmod{p}
  \end{equation}
    where the symbol $\ \hat{} \ $ indicates deleting the item. 
    
    Parts (i) and (ii) of the $M$-good conditions ensure that there exist at least two column vectors 
    $\begin{pmatrix}
    c_i \\
    d_i
    \end{pmatrix}$
    and
    $\begin{pmatrix}
    c_j \\
    d_j
    \end{pmatrix}$ 
    that are not parallel. Hence, we can always choose suitable indices \( i \) and \( j \) such that  
    \[
    \left| \begin{matrix}
    c_i &  c_j\\
    d_i & d_j \\
    \end{matrix} \right| \not\equiv 0 \pmod{p}.
    \] 
    
    By repeated application of Lemma \ref{lemma6.3} (lifting for $p \neq 3$), any solution of \eqref{lemma6.5_proof_3} (mod $p$) under the limit 
    $$
    x_i \not\equiv 0,\quad x_j \not\equiv 0,\quad(\text{mod }p)
    $$
    yields $p^{14(k-1)}$ solutions of \eqref{lemma6.5_proof_3} (mod $p^k$). So \eqref{lemma6.5_conclusion_1} in this case follows from
    \begin{equation*}
        \rho(p^k) \geq p^{14(k-1)}\rho_{nz}(p) 
    \end{equation*}
    and \eqref{lemma6.5_proof_1}.  
    
\end{proof}

Before establishing the result for the remaining case $p \leq p_0$, we introduce one auxiliary lemma for $p=3$.

\begin{lemma}[auxiliary lemma for $p=3$]
\label{beginning_mod_3}
The simultaneous equations 
    \begin{equation*}
        \begin{cases}
        c_1x_1^3+c_2x_2^3+c_3x_3^3 \equiv 0,\\
        d_1x_1^3+d_2x_2^3+d_3x_3^3 \equiv 0,
        \end{cases}
        \pmod{9}
  \end{equation*}
  have a zero $(x_1,x_2,x_3),$ satisfying at least two coordinates are not zero mod $3$, whenever $\frac{c_1}{d_1}, \frac{c_2}{d_2}, \frac{c_3}{d_3}, \pmod{3}$ are pairwise distinct.
\end{lemma}

\begin{proof}
This can be verified directly by the computer.

\end{proof}

\begin{lemma}[solutions satisfying lifting conditions exist under $M$-good]
\label{lifting_p_leq_p_0}
    Let $n\geq16$ and $p\leq p_0$, where $p_0$ is defined in Lemma \ref{lower_bound_rho(p^k)_p_geq_p_0}. Assume that the system of simultaneous equations \eqref{eq:zuikaishi} satisfies the $M$-good conditions (i), (ii) and (iii). Then 
    \begin{itemize}
        \item when $p \equiv1 \pmod{3}$, there exist $\gg p^{9}$ suitable solutions 
    $(x_1,...x_i,...,x_j,...,x_{16})$ of \eqref{lemma6.5_proof_2} with suitable $i,j$, satisfying  \eqref{lemma6.3_conclusion_1} holds;
    \item when $p \equiv2 \pmod{3}$, there exist $\gg p^{14}$ suitable solutions 
    $(x_1,...x_i,...,x_j,...,x_{16})$ of \eqref{lemma6.5_proof_2} with suitable $i,j$, satisfying  \eqref{lemma6.3_conclusion_1} holds;
    \item  when $p = 3$ there exists a suitable solution
    $(x_1,...x_i,...,x_j,...,x_{16})$
    of \eqref{lemma6.5_proof_2},$\pmod9$, with suitable $i,j$, satisfying \eqref{lemma6.4_conclusion_1} holds.
    \end{itemize}
\end{lemma}

\begin{proof}
    When $p \equiv1 ,\pmod{3}$, considering the part (i) of $M$-good when 
    there are at most nine identical elements in 
    $$S = \{ c_i*d_i^{-1}| \ i=1,...,16\}, \pmod{p}.$$
    Without loss of generality, we suppose that there are exactly nine identical elements and the indices are $1,...,9.$ Then we take 
    $(x_1,...,x_9)$  from $(\mathbb{Z}_p^*)^{9}$ freely and represent the remaining $x_{10},...,x_{16}$ through the following simultaneous equations
    \begin{equation}
    \label{lemma6.6_1}
    \begin{cases}
        c_{10}x_{10}^3+...+c_{16}x_{16}^3 \equiv -(c_{1}x_1^3+...+c_{9}x_9^3),\\
        d_{10}x_{10}^3+...+d_{16}x_{16}^3 \equiv -(d_{1}x_1^3+...+d_{9}x_9^3),\\
    \end{cases}
    \pmod{p}.
  \end{equation}
  By the Chevalley--Warning theorem, the system \eqref{lemma6.6_1} has a non-zero solution with respect to $x_{10},...,x_{16}$, whose non-zero coordinate, say $x_{10}$ and one of $x_1,...,x_9$ serve as the i-th and j-th coordinates of the desired solution, respectively. There are $\geq (p-1)^{9}\gg p^{9}$ choices.

When \(p \equiv 2 \pmod{3}\), consider part (ii) of $M$-good. There are at least three distinct ratios in \(S\) modulo \(p\). Once again, make use of the fact that every element in \(\mathbb{Z}_p\) is a cube. Thus, the system of simultaneous equations \(\eqref{lemma6.5_proof_2}\) is essentially linear, and the conclusion is straightforward. 
  
 When $p = 3$, the conclusion can be obtained directly from Lemma \ref{beginning_mod_3}.
  
\end{proof}

\begin{remark}
    As in the Remark \ref{remarek_1}, here we also only need to require that conditions (i), (ii) of the $M$-good property hold for primes \( p \leq M^2 \). Since we have assumed (after Lemma \ref{lemma2.1}) that no ratio \( r_i \) is repeated more than six times, we again aim to avoid the situation where a determinant is non-zero over the integers but becomes zero modulo \( p \).

    The part (i) of $M$-good has a different style from parts (ii) and (iii). When \(p\equiv1 \pmod{3}\), we stipulate the maximum number of the same ratios in order to apply the Chevalley–Warning theorem. Although this treatment is rather rough, setting similar conditions regarding the number of types of distinct ratios is not feasible for \(p\equiv1 \pmod{3}\). In fact, for the system of congruence equations \eqref{lemma6.6_1}, $p\equiv1 \pmod{3}$, even if we find five pairwise non-parallel column vectors, it is still not sufficient to ensure that there is a solution satisfying lifting condition \eqref{lemma6.3_conclusion_1}. Consider the following example:  
\[
\begin{cases}
x_1^3 + 2x_3^3 + 4x_4^3 + 6x_5^3 \equiv 0,\\
x_2^3 + 2x_3^3 + 2x_4^3 + 2x_5^3 \equiv 0,
\end{cases}
\pmod{7},
\]  
for which the only solution is $\mathbf{0}$.

    For $p = 3$, it is necessary to have at least three distinct types of ratios in $S$ modulo $3$ to guarantee the lifting process. Less kinds are not enough. To illustrate this, consider the scenario where, when the original equations degenerate into
    \begin{equation*}
        \begin{cases}
        c_{i}x_{i}^3+c_{j}x_{j}^3 \equiv 0,\\
        d_{i}x_{i}^3+d_{j}x_{j}^3 \equiv 0.\\
    \end{cases}
    \pmod{3}.
    \end{equation*}
    There are no solutions satisfying the lifting condition
    \begin{equation*}
        x_ix_j\left | \begin{matrix}
         c_{i}& d_{i} \\
          c_{j}& d_{j} \\
        \end{matrix} \right | \not\equiv 0\pmod{3}.
    \end{equation*}
\end{remark}

We define
\begin{equation*}
\rho_{\mathrm{lift}}(p^k) 
= \# \left\{ \mathbf{x} \in (\mathbb{Z}/p^k\mathbb{Z})^{16} : 
\begin{aligned}
&F(\mathbf{x}) \equiv 0 \pmod{p^k}, \quad G(\mathbf{x}) \equiv 0 \pmod{p^k}, \\
&\exists\, i, j \text{ such that } x_i x_j 
\left| \begin{matrix}
c_i & d_i \\
c_j & d_j \\
\end{matrix} \right| \not\equiv 0 \pmod{p}
\end{aligned}
\right\}.
\end{equation*}

By Lemma \ref{lifting_p_leq_p_0}, we have the following lemma.

\begin{lemma}[lower bound of $\rho(p^k)$ when $p< p_0$]
\label{lower_bound_rho(p^k)_p<p_0}
  If $n \geq 16,$ and the simultaneous equations  \eqref{eq:zuikaishi} are $M$-good, then\\
  for any $k \geq 1$ and any prime $3 \neq p < p_0 $, we have 
  \begin{equation*}
      \rho(p^k) \geq \rho_{\mathrm{lift}}(p^k) \gg_n p^{14k-5},
  \end{equation*}
    for any $k \geq 2$, we have 
  \begin{equation*}
      \rho(3^k) \geq \rho_{\mathrm{lift}}(3^k) \geq 3^{14k-28}.
  \end{equation*}
\end{lemma}

\begin{proof}
    Lower bounds of $\rho_{\mathrm{lift}} (p)$ and $\rho_{\mathrm{lift}}(3^2)$ come from Lemma \ref{lifting_p_leq_p_0}. 
    For higher power $k$, we can lift through the two lifting lemmas, Lemma \ref{lemma6.3} and Lemma \ref{lemma6.4}. So inductively,
    $$
    \rho_{\mathrm{lift}}(p^k) \geq p^{14(k-1)} \rho_{\mathrm{lift}}(p) \gg_n p^{14k-5},
    $$
    $$
    \rho_{\mathrm{lift}}(3^k) \geq 3^{14(k-2)} \rho_{\mathrm{lift}}(9) \gg_n 3^{14k-28}.
    $$  
\end{proof}

Now we can establish the lower bound of the truncated Euler product $S(P_0)$ defined by \eqref{S(P_0)}.

\begin{lemma}
\label{lemma_S_lower_bound}
  Let $\epsilon > 0$. Then we have
  \begin{equation}
  \label{S(P_0)_lower_bound}
      S(P_0) \gg_n \frac{1}{P_0^\epsilon}.
  \end{equation}
\end{lemma}

\begin{proof}
Recalling the definition \eqref{S(P_0)} of $S(P_0)$, 
we write $S(P_0)$ as
$$
S(P_0)=\prod_{2 \leq p < p_0}p^{-14k(p)}\rho(p^{k(p)}) \cdot \prod_{p_0 \leq p \leq P_0}p^{-14k(p)}\rho(p^{k(p)}).
$$
Invoking Lemma \ref{lower_bound_rho(p^k)_p_geq_p_0} and Merten's formula, we have
    \[
    \begin{aligned}
    \prod_{p_0 \leq p \leq P_0} p^{-14k(p)} \rho(p^{k(p)})
    &= \prod_{p_0 \leq p \leq P_0} \left(1 + O\left(\frac{1}{p}\right)\right) \\
    &\geq \prod_{p_0 \leq p \leq P_0} \left(1 - \frac{c}{p} \right) \\
    &\gg \left( \frac{\log p_0}{\log P_0} \right)^c \gg \frac{1}{P_0^\epsilon}.
    \end{aligned}
    \]
Invoking Lemma \ref{lower_bound_rho(p^k)_p<p_0}, we have
\[
\begin{aligned}
\prod_{2 \leq p < p_0} p^{-14k(p)} \rho(p^{k(p)})
&= 3^{-14k(3)} \rho(3^{k(3)}) \cdot \prod_{\substack{2 \leq p < p_0 \\ p \ne 3}} p^{-14k(p)} \rho(p^{k(p)}) \\
&\gg_n 3^{-28} \prod_{\substack{2 \leq p < p_0 \\ p \ne 3}} p^{-5} \\
&\gg_n 1.
\end{aligned}
\]
So \eqref{S(P_0)_lower_bound} follows.

\end{proof}

\begin{remark}
    In contrast, if we can only establish a weaker lower bound for \(\rho(p^k)\) for primes \(p\) in the range \(p_0 \leq p \leq P_0\), specifically,
\[
\frac{\rho(p^{k})}{p^k} = 1 + O\left(\frac{1}{p^{1 - \eta}}\right), \quad \text{for some } \eta > 0,
\]
then this is insufficient. In fact, we would have
\[
\begin{aligned}
\prod_{p_0 \leq p \leq P_0} p^{-14k(p)} \rho(p^{k(p)})
&= \prod_{p_0 \leq p \leq P_0} \left(1 + O\left(\frac{1}{p^{1 - \eta}}\right)\right) \\
&\leq \prod_{p_0 \leq p \leq P_0} \left(1 - \frac{c}{p^{1 - \eta}} \right) \quad  (\text{in the worst case}) \\
&\ll \exp\left( -c' \sum_{p_0 \leq p \leq P_0} \frac{1}{p^{1 - \eta}} \right).
\end{aligned}
\]
We can use Prime Number Theorem to deduce that
\[
\exp\left( -c' \sum_{p_0 \leq p \leq P_0} \frac{1}{p^{1 - \eta}} \right) \ll_{\eta, n} P_0^{-c''}, \quad \text{for any } c'' > 0.
\]
Later, we will obtain the equation \(\eqref{N(P)_last}\). We will find that the above upper bound is obviously unacceptable  if we compare the supposed main term with the supposed error terms. 
\end{remark}

Our next task is to approximate the truncated series $\mathfrak{S}(P_0)$ with the truncated Euler product $S(P_0)$. Define
\begin{equation*}
    R(P_0)=|\mathfrak{S}(P_0)-S(P_0)|.
    \end{equation*}

\begin{lemma}
\label{lemma_R_upper_bound}
  We have
  \begin{equation*}
      R(P_0) \ll M^{169+\frac{1}{3}} P_0^{-\frac{4}{3}+\epsilon}.
  \end{equation*}
\end{lemma}

\begin{proof}
    Recalling \eqref{truncated_S_definition}, \eqref{A(q)} and \eqref{S(P_0)}, we have
    \begin{equation*}
        R(P_0)=\left|\sum_{q \leq P_0}A(q)-\prod_{p \leq P_0}\sum_{i=0}^{k(p)}A(p^i)\right| \leq \sum_{q \in  \mathscr{L}(P_0)}|A(q)|,
    \end{equation*}
    where 
    $
    \mathscr{L}(P_0)=\{q > P_0: p|q \Rightarrow p \leq P_0\}
    $
    is a subset of $\{ q:q>P_0\}$.
    By \eqref{S(a,q)_upper_bound} and \eqref{lemma5.3_conclusion_2} , we have
    \begin{equation*}
    \begin{split}
         A(q)&=\mathop{\sum_{a_1=1}^{q}\sum_{a_2=1}^{q}}_{(a_1,a_2,q)=1} \prod_{i=1}^{16}\frac{S(b_i,q_i)}{q_i}\\
         &\ll \mathop{\sum_{a_1=1}^{q}\sum_{a_2=1}^{q}}_{(a_1,a_2,q)=1} \prod_{i=1}^{16}\frac{1}{q_i^{\frac{1}{3}}}\\
         &\ll q^{-\frac{7}{3}+\epsilon}M^{169+\frac{1}{3}}.
    \end{split}
    \end{equation*}
      So
      \begin{equation*}
        R(P_0) \ll \sum_{q >P_0}
        q^{-\frac{7}{3}+\epsilon}M^{169+\frac{1}{3}}
        \ll M^{169+\frac{1}{3}} P_0^{-\frac{4}{3}+\epsilon}.
    \end{equation*}  
\end{proof}

By Lemmas \ref{lemma_S_lower_bound} and \ref{lemma_R_upper_bound}, we have
\begin{equation*}
    \mathfrak{S}(P_0) \gg \frac{1}{P_0^\epsilon}-M^{169+\frac{1}{3}} P_0^{-\frac{4}{3}+\epsilon}.
\end{equation*}
Thus
\begin{equation}
\label{G_lower_bound}
    \mathfrak{S}(P_0) \gg \frac{1}{P_0^\epsilon} > 0
\end{equation}
whenever 
\begin{equation}
\label{singuler_sere_positive}
    P_0 \gg M^{128}.
\end{equation}

By Lemma \ref{lemma4.14}, Lemma \ref{lemma5.5}, Lemma \ref{lemma5.6}, and Lemma \ref{lemma5.8}, we have
\begin{equation}
\label{N(P)_last}
\begin{split}
    N (P )
    &=P^{\frac{16}{5}}\mathfrak{S}(P_0)I(P)
    +O\left(P^{\frac{43}{5}+3\delta}M^{\frac{19}{3}}P_0^{3}\right)\\    &+O\left(P^{\frac{41}{5}+\epsilon}M^{\frac{46}{3}}P_0^{\frac{11}{3}}\right)
    +O\left(P^{\frac{91}{10}+2\delta+\epsilon}M^{\frac{23}{3}+\epsilon}\right)\\    &+O\left(P^{\frac{46}{5}+4\delta+\epsilon}M^{186}P_0^{-\frac{4}{5}}\right)
    +O\left(P^{\frac{46}{5}-5\delta}M^{186}\right),
\end{split}
\end{equation}
where the $N (P )$ is given by \eqref{N(P)_definition}. 

Recalling the Lemma \ref{I(P)_lower_bound} and the inequality \eqref{G_lower_bound} , we have
\begin{equation*}
    P^{\frac{16}{5}}\mathfrak{S}(P_0)I(P) \gg P^{\frac{46}{5}-\epsilon} M^{-\frac{26}{3}},
\end{equation*}
under the assumptions \eqref{lemma6.1_1} and \eqref{singuler_sere_positive}.

To make $P^{\frac{46}{5}-\epsilon} M^{-\frac{26}{3}} \gg$ all five expected error terms in \eqref{N(P)_last}, we need
\begin{equation}
\label{bracket_1}
    P_0 \ll \frac{P^{\frac{1}{5}-\delta}}{M^{5+\epsilon}},
\end{equation}

\begin{equation}
\label{bracket_2}
    P_0 \ll \frac{P^{\frac{3}{11}}}{M^{\frac{72}{11}+\epsilon}},
\end{equation}

\begin{equation}
    \label{bracket_3}
P \gg M^{\frac{490}{3(1-20\delta)} + \epsilon},
\end{equation}

\begin{equation}
\label{bracket_4}
    P_0 \gg P^{5\delta+\epsilon}M^{730/3},
\end{equation}

\begin{equation}
\label{bracket_5}
    P \gg M^{\frac{584}{15\delta}+\epsilon},
\end{equation}
separately. And \eqref{singuler_sere_positive} can be ignored in comparison with \eqref{bracket_4};  \eqref{lemma6.1_1} can be ignored in comparison with \eqref{bracket_5}

Hence, the final result of Theorem \ref{zhudingli} is from the optimizations of 
\eqref{bracket_1},
\eqref{bracket_2},
\eqref{bracket_3},
\eqref{bracket_4},
\eqref{bracket_5}.

First, from \eqref{bracket_1} and \eqref{bracket_2} and under a modest assumption 
$
P \gg M^{30},
$
we take
\begin{equation}
\label{final_P_0}
    P_0=\min \left\{
    \frac{P^{\frac{1}{5}-\delta}}{M^{5+\epsilon}}, \frac{P^{\frac{3}{11}}}{M^{\frac{72}{11}+\epsilon}} \right\}
    =\frac{P^{\frac{1}{5}-\delta}}{M^{5+\epsilon}}.
\end{equation}
Note that $P_0$ also needs to be compatible with \eqref{bracket_4}, which requires satisfying $\delta <\frac{1}{30}$ and
$$
P \gg M^{\frac{3725}{3(1-30\delta)}+\epsilon}.
$$
Combining this with the left \eqref{bracket_3} and \eqref{bracket_5}, it suffices to set
\begin{equation}
\label{final_P}
    P=\max \left\{M^{\frac{490}{3(1-20\delta)} + \epsilon}, M^{\frac{584}{15\delta}+\epsilon}, M^{\frac{3725}{3(1-30\delta)}+\epsilon} \right\}.
\end{equation}
To optimize the three exponents in equation \eqref{final_P}, we take optimal values 
\[
\delta = \frac{584}{36145} \approx 0.0162, \quad P =M^{7229/3+\epsilon} \approx M^{2409.6667+\epsilon}.
\]

So 
$$P \gg M^{2409.6667+\epsilon}$$ 
ensures $N(P)>0$. Recall our definition \eqref{refion_R_definition} of the region $\mathscr{R}$; hence, 
$$ \Lambda_n \ll M^{2411}.$$ This concludes the proof of Theorem \ref{zhudingli}.

\section{The probability of meeting M-good}

It is difficult to find a way other than the circle method to address the alternative situation of the $M$-good condition.  Therefore, in this section, we focus on estimating the probability that a randomly chosen coefficient matrix satisfies the $M$-good condition, that is, to prove Theorem \ref{cidingli}. 

Recalling the definition of $M$-good (Definition \(\ref{def_M_good}\)), we now restate it in the language of coefficients matrices. Given the following integer matrix:
\[
\begin{pmatrix}
c_1 & c_2 & \cdots & c_{16} \\
d_1 & d_2 & \cdots & d_{16}
\end{pmatrix},
\]
and
$$
M=\max\{|c_1|,...,|c_{16}|,|d_1|,...,|d_{16}|\},
$$
we say the above matrix is $M$-good if and only if the following conditions hold:
\begin{itemize}
\item [(i)] for all $p \leq M^2$ with $p \equiv 1 \pmod{3}$, the matrix contains at most 9 pairwise parallel columns $\pmod{p}$;
\item [(ii)] for all $p \leq M^2$ with $p \equiv 2 \pmod{3}$, the matrix contains at least 3 distinct pairwise non-parallel columns $\pmod{p}$;
\item [(iii)] there exist at least 3 distinct ratios in $S \pmod{3}$.
\end{itemize}

We calculate the probabilities of parts (i), (ii), and (iii) separately, denoted as $\text{Prob}_1$, $\text{Prob}_2$, and $\text{Prob}_3$. Then, due to the mutual independence of (i), (ii), and (iii), the probability of the $M$-good condition is given as: 
$$
\text{Prob}_\text{$M$-good} = \text{Prob}_1\times\text{Prob}_2\times\text{Prob}_3.
$$

First, we calculate $\text{Prob}_3$. All cases of $(c_i,d_i)  \pmod{3}$ are shown in the following Table \ref{tab:Table1another}.

\begin{table}[h]
\caption{All cases of $(c_i,d_i)  \pmod{3}$.}
    \begin{tabular}{ccc}
    Ratio \((c_i \cdot d_i^{-1}) \bmod 3\) & Pairs \((c_i, d_i) \bmod 3\) & Probability of Each Pair \\
     \midrule
    \(0\) & \((0,1), (0,2)\) & \(2/9\) \\
    \(\infty\) & \((1,0), (2,0)\) & \(2/9\) \\
    \(1\) & \((1,1), (2,2)\) & \(2/9\) \\
    \(2\) & \((1,2), (2,1)\) & \(2/9\) \\
    \(\frac{0}{0}\) & \((0,0)\) & \(1/9\) \\
    \end{tabular}
 \label{tab:Table1another}   
\end{table}

The complementary event of $M$-good (iii) is having one or two distinct ratios in $S\pmod{3}$.

On the one hand, for one distinct ratio, we let \(k >0\) be the number of non-\((0,0)\) pairs in the sixteen positions. The case where the original matrix degenerates into a zero matrix is a separate situation. Choose one non-$\frac{0}{0}$ ratio from four choices and all $k$ non-$(0,0)$ pairs belong to it. Summing over $k = 1$ to $16$, the probability is 
\[
4\sum_{k = 1}^{16}\binom{16}{k}
\left(\frac{2}{9}\right)^k
\left(\frac{1}{9}\right)^{16 - k}
=\frac{4(3^{16}-1)}{3^{32}}.
\] 
The separate situation of zero matrix is $\left(\frac{1}{9}\right)^{16}$. So the total probability of one distinct ratio is
$\frac{4\cdot3^{15}-1}{3^{31}}$.

On the other hand, for two distinct ratios, we choose two non-\((0,0)\) ratios and their numbers are $k_1>0, k_2>0$, separately. The total probability of two distinct ratios is  
\[
\binom{4}{2}\sum_{k_1=1}^{16}\sum_{k_2=1}^{16-k_1}\binom{16}{k_1}\binom{16-k_1}{k_2}
\left(\frac{2}{9}\right)^{k_1}
\left(\frac{2}{9}\right)^{k_2}
\left(\frac{1}{9}\right)^{16-(k_1+k_2)}=\frac{6\times5^{16}}{9^{16}}-\frac{12}{3^{16}}+\frac{6}{9^{16}}.
\]   
Hence,  
\begin{equation}
\label{Prob_3_value}
        \text{Prob}_3=1-\frac{4(3^{16}-1)}{3^{32}}-\left(\frac{6\times5^{16}}{9^{16}}-\frac{12}{3^{16}}+\frac{6}{9^{16}}\right)=0.99951.
\end{equation}

Now we calculate $\text{Prob}_2$.
Analogously to the case of modulo 3, for a given prime number \(p\neq3\), we tabulate all the possible values of the ratio \(\frac{c_i}{d_i}\) (modulo \(p\)) along with their respective probabilities as shown in Table \ref{tab:ratio_and_prob}.
\begin{table}[htbp]
    \caption{Possible ratios $\frac{c_i}{d_i} \bmod{p}$ and their corresponding probabilities}
    \label{tab:ratio_and_prob}
    \begin{tabular}{cccccccc}
        $\frac{c_i}{d_i}$ & $0$ & $1$ & $2$ & $\cdots$ & $p - 1$ & $\infty$ & $\frac{0}{0}$ \\
        \midrule
        Probability & $\frac{p - 1}{p^2}$ & $\frac{p - 1}{p^2}$ & $\frac{p - 1}{p^2}$ & $\frac{p - 1}{p^2}$ & $\frac{p - 1}{p^2}$ & $\frac{p - 1}{p^2}$ & $\frac{1}{p^2}$ \\
    \end{tabular}
\end{table}
 
The complementary event of $M$-good (ii) is having one or two distinct ratios. 

On the one hand, for one distinct ratio, we let \(k >0\) be the number of non-\((0,0)\) pairs in the sixteen positions. The case where the original matrix degenerates into a zero matrix is a separate situation. Choose one non-$\frac{0}{0}$ ratio from $p+1$ choices and all $k$ non-$(0,0)$ pairs belong to it. Summing over $k = 1$ to $16$, the probability is 
\[
(p+1)\sum_{k = 1}^{16}\binom{16}{k}
\left(\frac{p-1}{p^2}\right)^k
\left(\frac{1}{p^2}\right)^{16 - k}
=\frac{p+1}{p^{32}}\sum_{k = 1}^{16}\binom{16}{k}(p-1)^k=\frac{(p+1)(p^{16}-1)}{p^{32}}.
\] 
The separate situation of zero matrix is $\left(\frac{1}{p^2}\right)^{16}$. So the total probability of 1 distinct ratios is
$$
\frac{(p+1)(p^{16}-1)}{p^{32}}+
\left(\frac{1}{p^2}\right)^{16}
=\frac{p^{16}+p^{15}-1}{p^{31}}.
$$

On the other hand, for two distinct ratios, we choose two non-\((0,0)\) ratios and their numbers are $k_1>0, k_2>0$, separately. The total probability of two distinct ratios is 
\begin{equation*}
\begin{split}
 &= \binom{p+1}{2} \sum_{k_1=1}^{16} \sum_{k_2=1}^{16-k_1} \binom{16}{k_1} \binom{16-k_1}{k_2} \left( \frac{p-1}{p^2} \right)^{k_1+k_2} \left( \frac{1}{p^2} \right)^{16-(k_1+k_2)} \\
&= \frac{1}{p^{32}} \binom{p+1}{2} \sum_{k_1=1}^{16} \sum_{k_2=1}^{16-k_1} \binom{16}{k_1} \binom{16-k_1}{k_2} (p-1)^{k_1+k_2} \\
&= \frac{p+1}{2p^{31}} \left( (2p - 1)^{16} - 2p^{16} + 1 \right).
\end{split}
\end{equation*}

To sum up, for a prime number \( p \) such that \( p \equiv 2 \pmod{3} \), the probability of the existence of three types of column vectors is given by  
\[
\text{Prob}_2(p) = 1 - f(p),
\]  
where  
\[
f(p) =  \frac{p^{16} + p^{15} - 1}{p^{31}}  +  \frac{p+1}{2p^{31}} \left( (2p - 1)^{16} - 2p^{16} + 1 \right).
\]
Here we get an upper bound of 
$\mathrm{Prob}_{M-\mathrm{good}}$,
\begin{equation}
\label{Prob_upper_bound}
    \mathrm{Prob}_{M-\mathrm{good}} \leq \text{Prob}_{2}(2)=0.9700.
\end{equation}

The probability of $M$-good (ii) is 
$$
\text{Prob}_2=\prod_{\substack{p\leq M^2\\p\equiv 2\bmod{3}}} \text{Prob}_2(p) > \prod_{p\equiv 2\bmod{3}}\text{Prob}_2(p).
$$
It is easy to find that \( f(p) \) satisfies the asymptotic estimate \( f(p) \ll \frac{1}{p^{14}} \), and there exists an absolute constant \( c \) independent of \( p \) such that \( f(p) \leq \frac{c}{p^{14}} \). In particular, through elementary calculations, the absolute constant \( c \) can be taken as \( 32767 \).

To ensure numerical precision, we decompose the product into two parts as follows:
\[
\text{Prob}_2>\prod_{\substack{p < p_i\\p\equiv 2\bmod{3}}}(1 - f(p))\cdot\prod_{p\geq p_i}\left(1-\frac{c}{p^{14}}\right),
\]
where $p_i$ denotes the $i$-th prime number.
To estimate the second product, we take the natural logarithm and apply a fact of the logarithmic function: specifically, for any $t\in[0,t_0]\subset[0,1)$, the inequality $\log(1 - t)\geq\frac{\log(1 - t_0)}{t_0}t$ holds. 
So we have
\begin{equation*}
\begin{split}
\prod_{p\geq p_i}\left(1-\frac{c}{p^{14}}\right)&\geq \exp\left(\log\left(1-\frac{c}{p_i^{14}}\right)\cdot\sum_{p\geq p_i}\left(\frac{p_i}{p}\right)^{14}\right).
\end{split}
\end{equation*}
Note that
$$
\log\left(1-\frac{c}{p_i^{14}}\right) > -\frac{c}{p_i^{14}}-\left(\frac{c}{p_i^{14}}\right)^2, \quad \sum_{p > p_i}\frac{1}{p_i^{14}} < \frac{1}{13p_i^{13}}.
$$
Consequently, we have
\begin{equation*}
    \begin{split}
        \log\left(1 - \frac{c}{p_i^{14}}\right)\cdot\sum_{p\geq p_i}\left(\frac{p_i}{p}\right)^{14} 
        &> -\frac{c}{p_i^{14}} - \left(\frac{c}{p_i^{14}}\right)^2 - c\sum_{p > p_i}\frac{1}{p^{14}} - \frac{c^2}{p_i^{14}}\sum_{p > p_i}\frac{1}{p^{14}}\\
        & =-\frac{c}{p_i^{14}} - \frac{c^2}{p_i^{28}}-\frac{c}{13p_i^{13}} - \frac{c^2}{13p_i^{27}}\\
        & >-\frac{c}{13p_i^{13}}.
    \end{split}
\end{equation*}
Moreover,
\begin{equation}
\label{prod_p_geq_pi_lower_bound}
\begin{split}
\prod_{p\geq p_i}\left(1-\frac{c}{p^{14}}\right)&\geq \exp\left(-\frac{c}{13p_i^{13}}\right)
\end{split}
\end{equation}
and
\begin{equation*}
\begin{split}
    \text{Prob}_2> \exp\left(-\frac{c}{13p_i^{13}}\right) \cdot \prod_{\substack{p < p_i\\p\equiv 2\bmod{3}}}(1 - f(p)).
\end{split}
\end{equation*}
We substitute $p_{100}=541$, thus  
\begin{equation}
\label{Prob_2_lower_bound}
    \text{Prob}_2> 0.969976831011652.
\end{equation}

Next, we calculate $\text{Prob}_1$.
For a prime number \( p \) such that \( p \equiv 1 \pmod{3} \), possible values of the ratio $\frac{c_i}{d_i}$ (modulo $p$) and their corresponding probabilities are also given by Table \ref{tab:ratio_and_prob}.

Let the number of occurrences of each ratio in Table \ref{tab:ratio_and_prob} among the sixteen ratios be \(k_0, k_1, \cdots, k_{p - 1}, k_{\infty}, k_{\frac{0}{0}}\), respectively. Note that the ratio \(\frac{0}{0}\) appears at most 8 times. Therefore, the vector $\mathbf{k}=(k_0, k_1, \cdots, k_{p - 1}, k_{\infty}, k_{\frac{0}{0}})$ that satisfies $M$-good (i) should lie within the region $\mathcal{U}$, determined by the following system of inequalities:
\[
\mathcal{U}:=\left\{
\begin{array}{l}
0 \leq  k_{\frac{0}{0}} \leq 8,\\
0 \leq k_0 + k_{\frac{0}{0}} \leq 9, \\
0 \leq k_1 + k_{\frac{0}{0}} \leq 9, \\
\vdots \\
0 \leq k_{p - 1} + k_{\frac{0}{0}} \leq 9, \\
0 \leq k_{\infty} + k_{\frac{0}{0}} \leq 9, \\
k_0 + k_1 + \cdots + k_{p - 1} + k_{\infty} + k_{\frac{0}{0}} = 16.
\end{array}
\right.
\]
For fixed $p$ and $\mathbf{k} = (k_0, k_1, \dots, k_{p-1}, k_{\infty}, k_{\frac{0}{0}})$ such that the sum of components is 16, the corresponding probability is given by:
$$
\begin{aligned}
&\binom{16}{k_0, k_1, \cdots, k_{p-1}, k_{\infty}, k_{\frac{0}{0}}} 
\left( \frac{p - 1}{p^2} \right)^{k_0} 
\left( \frac{p - 1}{p^2} \right)^{k_1} 
\cdots 
\left( \frac{p - 1}{p^2} \right)^{k_{p-1}} 
\left( \frac{p - 1}{p^2} \right)^{k_{\infty}} 
\left( \frac{1}{p^2} \right)^{k_{\frac{0}{0}}} \\
=& \binom{16}{k_0, k_1, \cdots, k_{p-1}, k_{\infty}, k_{\frac{0}{0}}} 
\left( \frac{p - 1}{p^2} \right)^{16 - k_{\frac{0}{0}}} 
\left( \frac{1}{p^2} \right)^{k_{\frac{0}{0}}} \\
=& \frac{(p - 1)^{16}}{p^{32}} 
\binom{16}{k_0, k_1, \cdots, k_{p-1}, k_{\infty}, k_{\frac{0}{0}}} 
\cdot \frac{1}{(p - 1)^{k_{\frac{0}{0}}}}.
\end{aligned}
$$
Thus the probability  that the matrix contains at most 9 pairwise parallel columns $\pmod{p}$ is
\begin{equation*}
    \text{Prob}_1(p)=\sum_{\mathbf{k} \in \mathcal{U}} \frac{(p - 1)^{16}}{p^{32}} \binom{16}{k_0, k_1, \cdots, k_{p - 1}, k_{\infty}, k_{\frac{0}{0}}} \frac{1}{(p - 1)^{k_{\frac{0}{0}}}}.
\end{equation*}
For better conformity with convention, we perform the following substitution:
\[
(k_0, k_1, \cdots, k_{p - 1}, k_{\infty}, k_{\frac{0}{0}}) \longrightarrow (n_1, \cdots, n_{p+1}, t).
\]
So
\begin{equation}
\label{Prob_1_p}
\begin{split}
\mathrm{Prob}_1(p) = 
&\sum_{t = 0}^{8}
\mathop{\sum_{0\leq t + n_1\leq 9} \cdots \sum_{0\leq t + n_{p + 1}\leq 9}}_{\substack{n_1 + \cdots + n_{p}+n_{p + 1}+t=16}}
\frac{(p - 1)^{16}}{p^{32}} 
\binom{16}{n_1, n_2, \cdots, n_{p}, n_{p + 1}, t} 
\frac{1}{(p - 1)^{t}}\\
=&\frac{(p - 1)^{16}}{p^{32}}\sum_{t = 0}^{8}\frac{1}{(p - 1)^{t}}\binom{16}{t}
\sum_{\substack{0\leq t + n_i\leq9,\\i = 1,\cdots,p + 1;\\n_1+\cdots + n_{p + 1}+t=16}}\binom{16-t}{n_1,n_2,\cdots,n_p,n_{p + 1}}.
\end{split}
\end{equation}

Calculating directly $\text{Prob}_1(p)$ is quite arduous, particularly when dealing with relatively large values of $p$. For $p>7$, we alternatively calculate the complementary probability that a column vector appears at least ten times, denoted as $\overline{\text{Prob}_{1}(p)}$.  We select ten specific positions from sixteen, where the \(a\) positions correspond to one of the \(p + 1\) non- \(\frac{0}{0}\) ratios, and the remaining \(b\) positions correspond to the \(\frac{0}{0}\) ratio, satisfying \(a + b = 10\).
So we get an upper bound of $\overline{\text{Prob}_{1}(p)}$.
\begin{equation*}
\begin{split}
    \overline{\text{Prob}_{1}(p)} 
    &\leq \binom{16}{10} \sum_{\substack{a+b=10}} \left(\left(\frac{p-1}{p^2}\right)^a \left(\frac{1}{p^2}\right)^b  \binom{p+1}{1}\right)\\
    &=\binom{16}{10} \frac{p+1}{p^{20}}\sum_{a+b=10}(p-1+1)^{10}\\
    &=\binom{16}{10}\frac{p+1}{p^{10}} \leq 9152\cdot \frac{1}{p^9}.
\end{split}
\end{equation*}
We have
\begin{align*}
\text{Prob}_1 &= \text{Prob}_1(7) \cdot \prod_{\substack{p\equiv1\pmod{3}\\p\geq13}} \text{Prob}_1(p)\\
&= \text{Prob}_1(7) \cdot \prod_{\substack{p\equiv1\pmod{3}\\p\geq13}} \left(1 -\overline{\text{Prob}_{1}(p)}\right)\\
&\geq \text{Prob}_1(7) \cdot \prod_{\substack{p\equiv1\pmod{3}\\p\geq13}} \left(1 - \frac{9152}{p^9}\right).
\end{align*}
By \eqref{Prob_1_p}, $\text{Prob}_1(7)=0.99990129$.
On the other hand, by \eqref{prod_p_geq_pi_lower_bound}, we have
\begin{equation*}
    \begin{split}
        \prod_{\substack{p\equiv1\pmod{3}\\p\geq13}} \left(1 - \frac{9152}{p^9}\right) 
        &\geq \prod_{p\geq13} \left(1 - \frac{9152}{p^9}\right)\\
        &\geq \frac{1}{\exp{\frac{9152}{13\cdot 13^{13}}}}\\
        &=0.9999999999977.
    \end{split}
\end{equation*}
Consequently,
\begin{equation}
    \label{Prob_1_lower_bound}
    \text{Prob}_1 \geq 0.9999.
\end{equation}

By synthesizing the results from 
\eqref{Prob_3_value},
\eqref{Prob_upper_bound}, \eqref{Prob_2_lower_bound}, and \eqref{Prob_1_lower_bound},  we obtain
$$
0.9694\leq \text{Prob}_\text{$M$-good} \leq  0.9700.
$$

\begin{acknowledgements}
    The authors would like to thank Professor Igor Shparlinski for his very helpful suggestions and comments on an earlier draft of this paper.
\end{acknowledgements}

\begin{funding}
  This work was supported by National Natural Science Foundation of China (Grant
No.12171311).
\end{funding}


\end{document}